\documentclass[12pt]{article}

\usepackage[utf8]{inputenc}
\usepackage{amsmath}
\usepackage{amsfonts, amssymb}
\usepackage{graphicx}
\usepackage{caption}
\usepackage{subcaption}
\usepackage{blindtext}
\usepackage{cite}
\usepackage{amsbsy}
\usepackage{indentfirst}
\usepackage{fullpage}
\usepackage{pstricks}
\usepackage{faktor}
\usepackage{float}
\usepackage{datetime}

\usepackage{yfonts}
\usepackage{amsfonts}
\usepackage{mathrsfs}
\usepackage{hyperref}
\usepackage{datetime}
\usepackage[ampersand]{easylist}
\usepackage{caption}
\usepackage{upgreek}
\usepackage{theoremref}
\usepackage{tikz-cd}
\usepackage{amsthm}
\usepackage[normalem]{ulem}

\theoremstyle{definition}
\newtheorem{theorem}{Theorem}
\newtheorem{corollary}[theorem]{Corollary}
\newtheorem{proposition}[theorem]{Proposition}
\newtheorem{lemma}[theorem]{Lemma}
\newtheorem{remark}[theorem]{Remark}
\newtheorem{definition}[theorem]{Definition}
\newtheorem{ex}[theorem]{Example}
 \newcommand{\beq}{\begin{equation}}
\newcommand{\eeq}{\end{equation}}

\newcommand{\sym}{\text{sym}}

\def\R{\mathbb R}
\def\C{\mathbb C}
\def\Z{\mathbb Z}
\def\bc{c}
\def\hc{\hat c}
\def\tA{\tilde A}
\def\tC{\tilde C}
\def\al{\alpha}
\def\v{{\bf v}}
\def\ka{\kappa}
\def\gam{\gamma}

\def\cgam{\mathcal{C}}

\def\ep{\varepsilon}
\def\limn{\lim_{n\to\infty}}
\definecolor{darkmagenta}{rgb}{0.55, 0.0, 0.55}
\definecolor{olivedrab}{rgb}{0.42, 0.56, 0.14}
\definecolor{darkblue}{rgb}{0.0, 0.0, 0.55}
\definecolor{brown}{rgb}{0.65, 0.16, 0.16}
\newcommand\seq[1]{\left\{#1\right\}}
\newcommand\maxc[1]{\left<#1\right>}
\newcommand\infin[1]{\left|\left|#1\right|\right|_{[0,L]}}

\title{Euclidean and Affine Curve  Reconstruction}
\date{}
\begin{document}
\maketitle
\begin{center}
\begin{tabular*}{.95\textwidth}{@{\extracolsep{\fill}} lll}
Jose Agudelo\footnotemark[1] & Brooke Dippold\footnotemark[2]  & Ian  Klein\footnotemark[3]\\
University of New Mexico&Longwood University & NC State University \\
joseagudelo@unm.edu & brookedippold1@gmail.com&iklein@ncsu.edu\\
orcid:
0000-0001-5052-2112&orcid: 0000-0001-9274-6704 &orcid:0000-0003-1140-226X
  \\
&&\\
Alex Kokot\footnotemark[4] & Eric Geiger \footnotemark[5]   & Irina Kogan\footnotemark[6]\\
University of Washington&{ Baruch College, CUNY} &NC State University \\
 	akokot@uw.edu & { eric.geiger@baruch.cuny.edu} & {iakogan@ncsu.edu} \\
	orcid: 0000-0001-6163-6051&orcid:0000-0002-5296-0199&orcid: 0000-0001-8212-6296
\end{tabular*}

\end{center}
\footnotetext[1]{{\it Jose Agudelo was an undergraduate at  ND State University when this REU project was performed.}}
\footnotetext[2]{{\it Brooke Dippold was an undergraduate at Longwood University when this REU project was performed. }}
\footnotetext[3]{{\it Ian Klein was an undergraduate at Carleton College when this REU project was performed. }}
\footnotetext[4]{{\it Alex Kokot was an undergraduate at the University of Notre Dame when this REU project was performed. }}
\footnotetext[5]{{\it Eric Geiger was a graduate student at NC State University and a mentor for this REU project. }}
\footnotetext[6]{{\it Irina Kogan is a Professor of Mathematics at NC State University and a  mentor for this REU project.  }}
 \abstract{We consider practical aspects of reconstructing planar curves with prescribed Euclidean or affine curvatures. These curvatures are invariant under the special Euclidean group and the special affine groups, respectively, and play an important role in computer vision and shape analysis. We discuss and implement  algorithms for such reconstruction, and  give estimates on how close reconstructed curves  are relative to the closeness of their curvatures in   appropriate metrics. Several illustrative examples are provided.}

\vspace{.5cm}

\section{Introduction}\label{sect-intro}
Rigid motions -- compositions of translations, rotations and reflections -- are fundamental transformations on the plane studied in a high-school  geometry course. Two shapes related by these transformations are called \emph{congruent}. The geometry studied in  high-school  is based on  the set of axioms formulated by Euclid around 300BC and is called  \emph{Euclidean geometry}. Rigid motions make up the set of all transformations on the plane that preserve Euclidean distance between two points. A composition of two rigid motions is again a rigid motion, and the set of all rigid motions with the binary operation defined by composition satisfies the definition of a group (see Section~\ref{ssect-cong}). Naturally, this group  is called the \emph{Euclidean group} and  is denoted by $E(2)$, where $2$ indicates that the motions are considered in the 2-dimensional space, the plane.

 To a human eye, two figures look the same if they are related  by a rigid motion. However, since a  reflection changes the orientation of an object, a group of orientation-preserving rigid motions, consisting of rotations and translations only,  is often considered. This group is called the  \emph{special Euclidean group} and is denoted by $SE(2)$. In many applications, the congruence  with respect to other groups is considered. For example, two shadows cast by the same object onto two different  planes  by blocking the rays of light emitted from a lamp are related by a projective transformation.  If a light source  can be considered to be infinitely far away (like a sun),  then the shadows are related by an affine transformation. See  \cite{hartley04} for an excellent exposition of the roles played by projective, (special) affine,  and (special) Euclidean transformations in computer vision. Starting in the 19th century, it was widely accepted that Euclidean geometry, although the most intuitive,  is not the only possible consistent geometry, and that   congruence can be defined relative to other transformation groups \cite{hawkins-1984}.

 \begin{figure}
  \begin{minipage}{.45\linewidth}
  \centering
  \includegraphics[width=\linewidth]{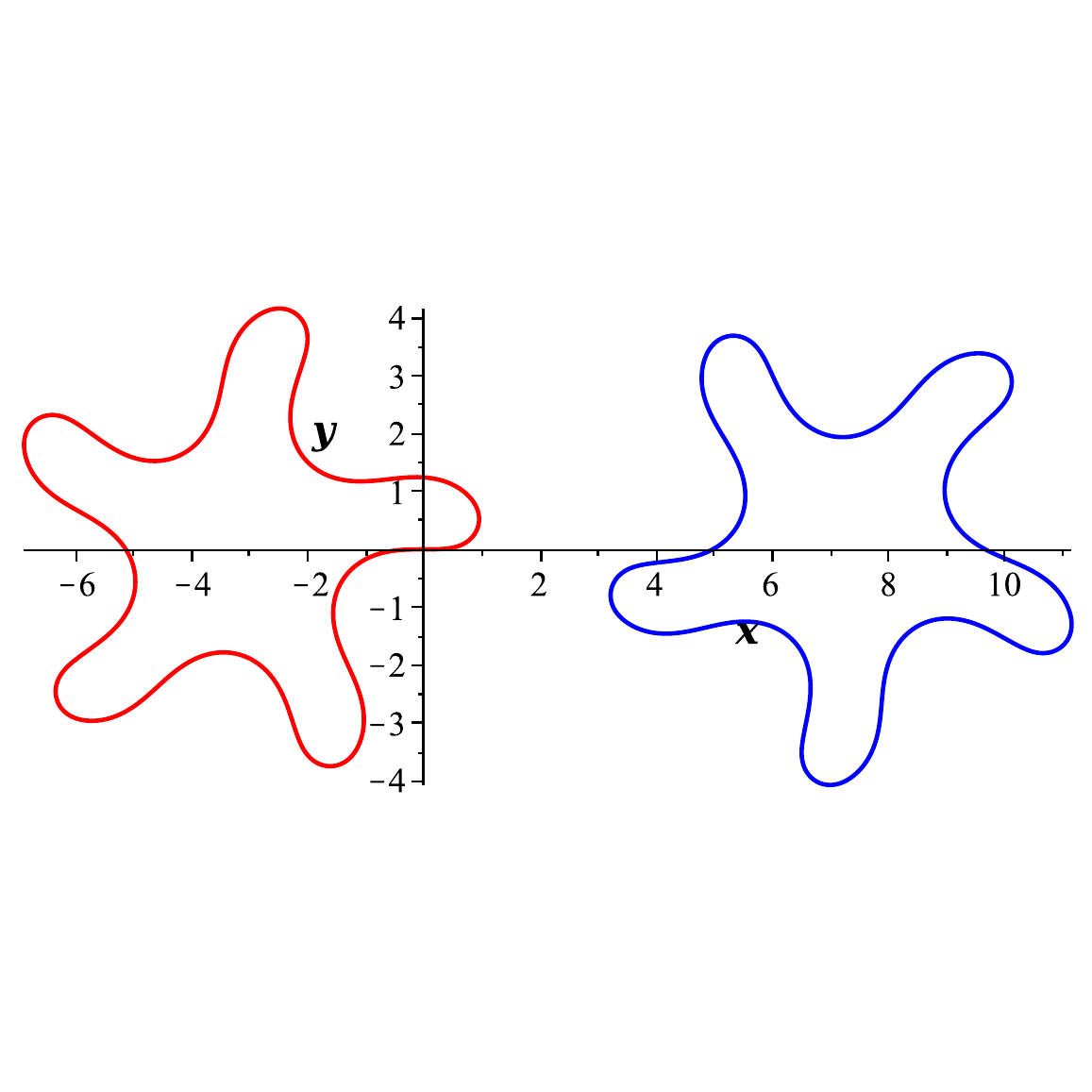}
  \caption{A special Euclidean transformation is a composition  of a rotation and a translation.}
  \label{fig-euc-trans}
  \end{minipage}
  \hspace{.05cm}
  \begin{minipage}{.45\linewidth}
  \centering
  \includegraphics[width=\linewidth]{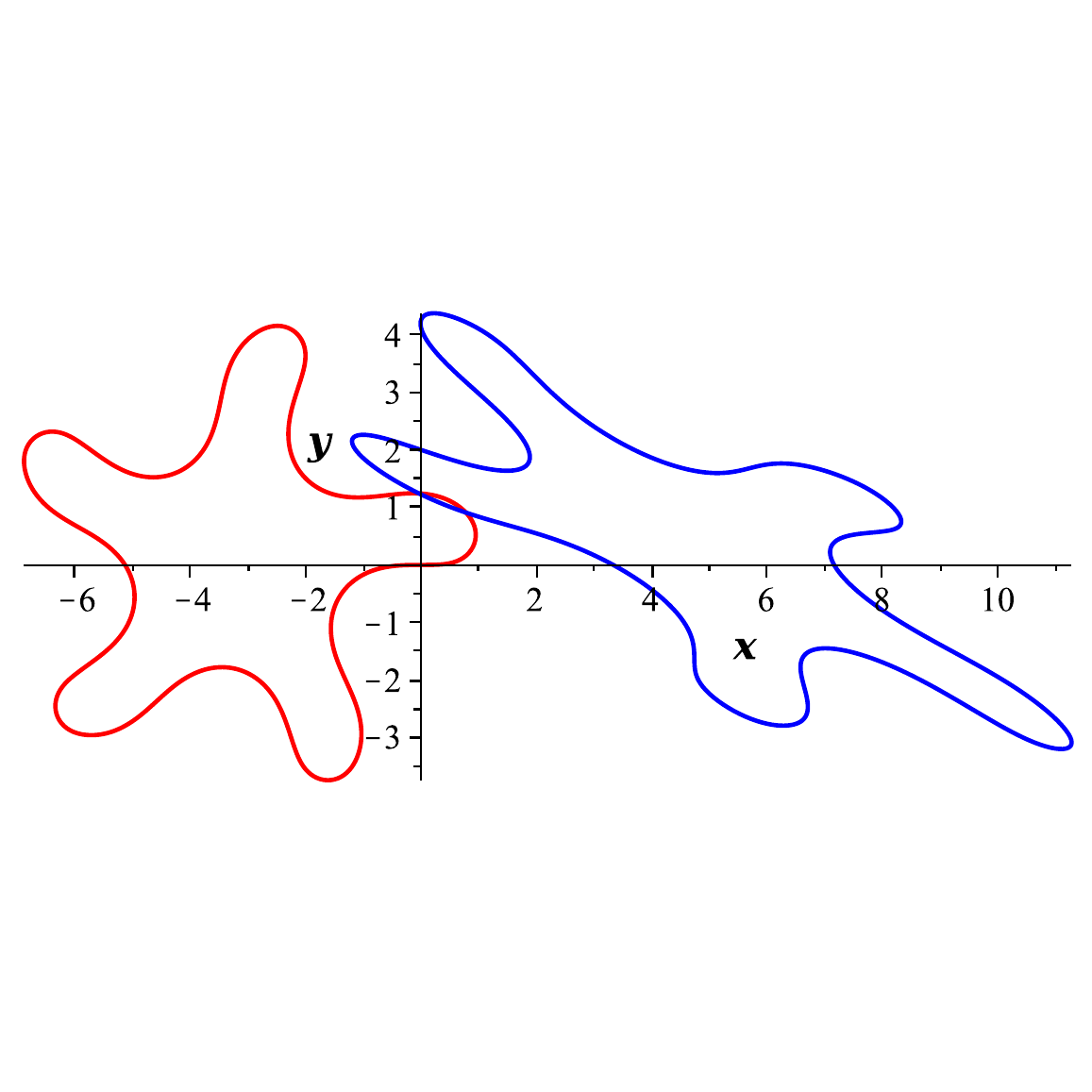}
  \caption{A special affine transformation is a composition  of a unimodular  linear transformation and a translation.}
  \label{fig-affine-trans}
  \end{minipage}
\end{figure}

 In this work, we consider  congruence of planar curves relative to the special Euclidean group $SE(2)$ and the \emph{special affine group} $SA(2)$.  The latter group consists of compositions of area and orientation  preserving (i.e.~\emph{unimodular}) linear   transformations and translations, and is sometimes also called the \emph{equi-affine group}.  In Figure~\ref{fig-euc-trans}, we show two curves related by a special Euclidean transformation, while in Figure~\ref{fig-affine-trans} we show two curves related by a  special affine transformation. For applications of curve matching under (special) Euclidean and affine transformations see, for instance, \cite{wolfson88, faugeras-1994, COSTH-1998, shakiban02, goldberg2004, GMW-10, FH-2007, HO14}.

  It is widely known that two sufficiently smooth  planar curves are $SE(2)$-congruent if they have the same Euclidean curvature $\kappa$ as a function of the Euclidean arc-length $s$. Somewhat less familiar, but also known from the 19th century,  are the notions of curvature and arc-length in other  geometries, in particular in the special affine geometry, \cite{gugg}. Similarly to the Euclidean case, one can show that    two sufficiently smooth  planar curves are $SA(2)$-congruent if  they have the same affine curvature $\mu$ as a function of the affine arc-length $\al$. 
 Knowing that  the curvature as a function of the arc-length determines a curve up to the relevant group of transformations, it is natural to ask two questions:
 \begin{enumerate}
 \item Is  there a practical algorithm to reconstruct a curve from its curvature up to the relevant transformation group?
 \item If two curvatures are close to each other in a certain metric, how close can the reconstructed curves be brought  to each other by an element of  the relevant transformation group?
 \end{enumerate}
 In this paper, we study both of these questions, by methods and techniques that are well known. Namely, we review and implement a procedure  for reconstructing  a curve from its Euclidean curvature by successive integrations. The procedure for reconstructing  curves from its affine curvature is more complicated and is based on Picard iterations. An implementation of these procedures can be found  at \url{https://egeig.com/research/curve_reconstruction}.  In Theorem~\ref{thm-euc-est}, we show how close, relative to the Hausdorff metric, two curves can be brought together by a  special Euclidean transformation if their Euclidean curvatures are $\delta$-close in  the $L^\infty$-norm. 
 Theorem~\ref{thm-aff-est} addresses the same question in the special affine case.  
 
Many of the theoretical  results presented in this paper are well known and the new results presented here are hardly surprising. However,  combined together and illustrated by specific examples, we believe, they contribute to a better understanding of a classical, but important problem, relevant in many modern applications.
 This paper is the result of an REU project, which turned out to be of great pedagogical value, as it taught the students to combine the results  and methods from  various subjects: differential geometry, algebra, analysis and  numerical analysis. In addition, this project involved  theoretical work and   the work of designing and implementing algorithms.  
 The multidisciplinary  nature of this project, on one hand, and its accessibility, on the other hand,  allowed the undergraduate participants to truly experience the richness and challenges of  mathematical research. We hope  that  we are able to convey to the reader the enjoyment of various aspects of the mathematical research that we experienced while working on the project.
 
The paper is structured as follows. Section~\ref{sect-prelim} contains preliminaries and is split in the following subsections. In Section~\ref{ssect-cong}, { after  reviewing the definitions of groups and group actions, we} define the notions of congruence and symmetry of curves relative to a given group.  In Sections~\ref{ssect-emf}
 and~\ref{ssect-amf}, we follow \cite{gugg} to define Euclidean and affine moving frames and invariants. 
  In Section~\ref{ssect-norms}, { we introduce norms and distances, used in this paper, in the spaces of functions, matrices, matrices of functions, and curves and prove  some useful inequalities}.    In Section~\ref{ssect-conv}, we establish some results about convergence of matrices and their norms. 
 
Section~\ref{sect-erecon}  contains explicit formulas for  reconstructing a curve from its Euclidean curvature function and gives an upper bound on the closeness of reconstructed curves with close Euclidean curvatures. Section~\ref{sect-arecon} introduces a Picard iteration scheme for  reconstructing a curve from its affine curvature function and gives an upper bound on the closeness of reconstructed curves with close affine curvatures. Directions of further research are indicated in Section~\ref{sect-concl}. In the appendix, we derive power a power series representation for curves whose affine curvatures are given by a monomial. 
 

\section{Preliminaries}\label{sect-prelim}

\subsection{Congruence and symmetry of the planar curves}  \label{ssect-cong}

To keep the presentation self-contained, we remind the reader the standard definitions of groups and group-actions. 
\begin{definition}\label{def-grp}
A \emph{group} is a set $G$ with a binary operation $``\cdot"\colon  G\times G\to G$ that satisfies the following properties:
\begin{enumerate}
    \item  \emph{Associativity:}  $ (g_1 \cdot  g_2) \cdot g_3 = g_1 \cdot (g_2 \cdot g_3)$, $\forall g_1, g_2, g_3 \in G.$  
\item\emph{Identity element:} 	There exists a unique $e \in G$, such that $e \cdot  g= g \cdot e = g$, $\forall  g \in G$.  
\item \emph{Inverse element:} 	For each $g \in G$, there exists an element $h\in G$ such that $g \cdot h = h \cdot  g = e$. 
We denote $g^{-1}:=h.$
\end{enumerate}
\end{definition}
\begin{definition}\label{def-act}
An  \emph{action}  of a group $G$  on a set $P$ is a map $\phi\colon  G\times P\to P$  satisfying the following properties:

\begin{enumerate}
    \item  \emph{Associativity:}  $ \phi(g_1 \cdot  g_2, p) = \phi(g_1, \phi (g_2, p ))$, $\forall g_1, g_2\in G$  and  $\forall p\in P$
\item\emph{Action of the identity element:} 	 $\phi(e,p) = p$, $\forall  p \in P$.  \end{enumerate}
We use a shorter notation  $\phi(g,p):=gp$.   Each element $g\in G$ determines a bijective map $g\colon P\to P$, $p\to gp$.
\end{definition}
Groups are often defined through their actions.  For example, a
  \emph{rotation}  in the plane 
by angle $\theta>0$ about the origin  in  the counter-clockwise direction sends a point $(x,y)$ in the plane to a point
\beq\label{r-act} (\bar x, \bar y)=(x\cos\theta-y\sin\theta, x\sin\theta+y\cos\theta)=(x,y)R_\theta^{-1},\eeq
where  the $2\times 2$ matrix $R_\theta$ is given by  
\beq\label{r-matrix}R_\theta=\begin{pmatrix} \cos \theta&-\sin \theta\\ \sin\theta&\cos\theta\end{pmatrix}.\eeq
We multiply by the matrix on the right because we treat points (and vectors) in $\R^2$ as \emph{row vectors}. We invert the matrix to satisfy the associativity property in the definition of the group action \footnote{Since rotation matrices commute, the associativity property will be satisfied without the inversion, but  it is essential for generalizations to other groups.}.
Rotation by $\theta=0$ corresponds to the identity matrix and leaves all points in place, while  $R_\theta$ with $\theta<0$ corresponds to the clockwise rotation by the angle $|\theta|$.
The set of matrices $\{R_\theta\,|\,\theta\in \R\}$ with the binary operation given by matrix multiplication satisfies the definition  of a group. This group is called the \emph{special orthogonal group}  and is denoted by $SO(2)$. The word \emph{special} in the name of the group indicates that $\det(R_\theta)=1$ and so the orthonormal basis defined by its columns (or rows) is positively oriented. In fact,  the group $SO(2)$ consists  of all $2\times 2$ matrices whose
 two columns (or two rows) form a positively oriented orthonormal basis in $\R^2$.
The map $\phi\colon SO(2)\times\R^2\to\R^2$defined by \eqref{r-act} satisfies the definition of a group-action. The associativity property in Definition~\ref{def-act} states that the action of the product of matrices $R_{\theta_1}\cdot R_{\theta_2}$  is the composition of the rotation by the angle $\theta_2$ followed by the rotation by the angle  $\theta_1$.  

The \emph{translation} in the plane by a vector $\v =(a,b)$
 sends a point  $(x,y)$ to the point
 \beq\label{t-act}(\bar x, \bar y)=(x,y)+(a,b)=(x+a, y+b).
\eeq
The set of vectors $\v \in\R^2$ with the binary operation given by vector addition satisfies the definition  of a group, with the zero vector being the identity element of this group. Formula \eqref{t-act} describes the action of this group on the plane.
The composition of a rotation by $\theta$ followed by a translation by $\v$ sends a point  $(x,y)$ to the point
\beq\label{e-act}(\bar x, \bar y)=(x\cos\theta-y\sin\theta+a, x\sin\theta+y\cos\theta+b)=(x,y)R_\theta^{-1}+\v.\eeq
 The set of all compositions of rotations and translations also satisfies the definition of a group. It is called the \emph{special Euclidean group} and is denoted $SE(2)$. This is the  group of all transformations in the plane that preserves distances (and, therefore, angles) in the plane, as well as the orientation.  
The composition of a rotation/translation pair $(R_{\theta_2}, \v_2)$ followed by a pair $(R_{\theta_1}, \v_1)$ is equivalent to  the rotation $R_{\theta_1}R_{\theta_2}=R_{\theta_1+\theta_2}$ followed by  the translation by the vector $\v_2R_{\theta_1}^{-1}+\v_1$. Thus we can think of the special Euclidean group  as the set of pairs $\{(R_{\theta}, \v)\,|\, R_\theta\in SO(2), \v\in\R^2\}$ with the group operation 
\beq\label{e-gp}(R_{\theta_1}, \v_1)\cdot(R_{\theta_2}, \v_2)=\left(R_{\theta_1}R_{\theta_2},\v_2 R_{\theta_1}^{-1}+\v_1\right).
\eeq
In other words, $SE(2)=SO(2)\ltimes\R^2 $ is a semi-direct product of the  translation and  rotation groups.

If in \eqref{e-act} and \eqref{e-gp}, we replace  the rotation matrix  $R_\theta$ with an arbitrary non-singular  $2\times 2$ matrix $M$, we obtain an action of the \emph{affine group},     $A(2)=GL(2)\ltimes\R^2 $, a semi-direct product of the group   of invertible linear transformations and translations.  Restricting the matrix $M$ to the group of  \emph{unimodular} matrices   $SL(2)=\{M\,|\, \det(M)=1\}$, we obtain a smaller group which is called the \emph{special affine} or the  \emph{equi-affine} group\footnote{From now on we will use the term   \emph{special affine}.} $SA(2)=SL(2)\ltimes\R^2$. A generic $SA(2)$-transformation does not preserve distance or angles, but it preserves areas.

 An action of a group on the plane induces the action on the curves in the plane. In this paper, we consider curves satisfying the following definition.
\begin{definition}[Planar curve]  A planar \emph{curve} $\cgam$ is the image of a \emph{continuous locally injective}\footnote{A map $\gamma\colon I \to \R^2$, where $I$ is an open  subset of $\R$ is \emph{locally injective} if for any $t\in I$, there exists an open neighborhood $J\subset I$, such that $\gamma|_J$ is injective.}  map $\gamma\colon \R\to\R^2$.   We  call $\cgam$ \emph{closed} if its parameterization $\gamma$ is periodic. We often restrict the domain of $\gamma$ to an open or a closed interval  $I\subset \R$.
\end{definition} 
  Given a group $G$ acting continuously on the plane,  the image of a curve $\cgam$ parametrized by $\gamma$, under a transformation $g\in G$ is the curve
  $g\cgam=\{gp\,|\,p\in\cgam\}$ parametrized by $g\gamma=g\circ\gamma$.
\begin{definition} Given a group $G$ acting on the plane, we say that two planar curves $\cgam_1$ and $\cgam_2$ are \emph{G-congruent} ($\cgam_1\underset{G}\cong\cgam_2$) if there exists $g\in G$, such that 
 $\cgam_2=g\,\cgam_1$.
\end{definition}
\begin{definition} An element $g\in G$ is a \emph{$G$-symmetry of $\cgam$} if 
$$g\, \cgam=\cgam.$$ 
It easy to show  that the set of such elements, denoted $\sym_G(\cgam)$,   is a subgroup of $G$, called the  \emph{$G$-symmetry group of $\cgam$}. 
The cardinality of $\sym_G(\cgam)$ is called the \emph{symmetry index} of $\cgam$.
\end{definition}

In Figure~\ref{fig-euc-trans}, we show two $SE(2)$-congruent curves, each with five $SE(2)$-symmetries. In Figure~\ref{fig-affine-trans}, we show two $SA(2)$-congruent curves, each with five $SA(2)$-symmetries. As a side remark, we  note that the five $SA(2)$-symmetries of the left curve in Figure~\ref{fig-affine-trans}, in fact, belong to $SE(2)$, while the five $SA(2)$-symmetries of the right curve do not. 
The method of moving frames, pioneered by Bartels, Frenet, Serret, Cotton, and Darboux,  and greatly extended by Cartan, allows to solve the $G$-congruence problem for sufficiently smooth curves\footnote{A curve is called $C^k$-smooth if  the $k$-th order derivative of its parametrization $\gamma$ is continuous.}  by assigning a frame of basis vectors along a curve, in a way that is compatible with the $G$-action. We will review this classical construction of such frames for the $SE(2)$ and $SA(2)$ actions by following, for the most part, the exposition given in \cite{gugg}. For a more detailed   history and  generalizations to arbitrary Lie group-actions  see \cite{olver2015}.
\subsection{Euclidean moving frame and invariants}  \label{ssect-emf}

The $SE(2)$-frame at point  $p$ of a planar curve $\cgam$ consists of the unit tangent vector  $T(p)$ and  the unit normal vector $N(p)$. Orientation for $T(p)$ is defined by the parameterization $\gamma$ of  $\cgam$, while the orientation for 
$N(p)$ is chosen so that the pair of vectors $T(p)$ and $N(p)$  is positively oriented, i.e.~the closest rotation from $T(p)$ to $N(p)$ is counter-clockwise. Considering $T$ and $N$ to be row vectors, we combine  them into an \emph{$SE(2)$-frame matrix}
\beq \label{A}A_{\cgam}(p)=
\begin{pmatrix}{}
    T(p) \\
  N(p)  \\
\end{pmatrix}.
\eeq
An important observation is that $A_{\cgam}(p)$ is an orthogonal matrix. In fact,  it is precisely the rotation matrix which brings the moving frame basis consisting of $T(p)$ and $N(p)$ to  the standard orthonormal basis in $\R^2$ under the action \eqref{r-act}. 
 An element $g\in SE(2)$ acting on $\R^2$ maps the curve $\cgam$ to $\tilde\cgam$ and  the point $p$ to $\tilde p$. Since the   $SE(2)$-action preserves tangency and length, it maps the $SE(2)$-frame at $p\in\cgam$  to   the $SE(2)$-frame at $\tilde p\in\tilde \cgam$. 
 See Figure~\ref{fig:eucl.mf} for an illustration.
This compatibility property of the frame is called \emph{equivariance} and can be expressed as 
\beq\label{frame-equiv}A_{g\,\cgam}(g p)=A_{\cgam}(p)R_g^{-1},\eeq
where $\cgam$ is an arbitrary curve, $p\in \cgam$,  $g\in SE(2)$, $R_g$ is the rotational part of the transformation $g$.

It is well known that any $C^1$-smooth non-degenerate curve $\cgam$ can be parametrized
$$\gamma\colon s\to p=\gamma(s),$$ so that 
\beq\label{e-gs}T(p)=\gamma_s(s)\eeq is the unit tangent vector at the point $p=\gamma(s)\in \cgam$. (Here and below, a variable in the subscript denotes the differentiation with respect to this variable). Explicitly, if $\hat \gamma\colon t\to\gamma(t)=(x(t),y(t))$ is any parametrization of $\cgam$, then
$$s(t)=\int_0^t|\hat \gamma_\tau|d\tau=\int_0^t\sqrt{x'(\tau)^2+y'(\tau)^2}d\tau.$$
The parameter $s$ is called the \emph{Euclidean arc-length} parameter. Its differential
\beq\label{ds} ds=  |\hat \gamma_t|dt\eeq
is called the infinitesimal \emph{Euclidean arc-length}.
Clearly, the integral of $ds$ along  a curve segment produces the Euclidean length of the  curve-segment.
 
 We now assume that $\cgam$ is $C^2$-smooth and note that the differentiation of the identity  $|T(s)|=1$ implies that $T_s(s)$ is orthogonal to $T(s)$, and so  $T_s(s)$ is proportional to $N(s)$. Thus there is a function $\kappa(s)$, called the \emph{Euclidean curvature function}, such that
 \beq\label{e-ts}T_s(s)=\kappa(s) N(s).\eeq 
Explicitly,   $\kappa(s)=\pm|\gamma_{ss}|$, with $``+"$ when the rotation from $\gamma_s$ to $\gamma_{ss}$ is counterclockwise and $``-"$ otherwise. For an arbitrary parameterization $\hat\gamma(t)$, we have
\beq\label{kappa} \kappa(t)=\frac{\det(\hat \gamma_t, \hat \gamma_{tt})}{|\hat \gamma_t|^3}.
\eeq
The Euclidean curvature of a circle of radius $r$ is constant and is equal to $\frac 1 r$. The  Euclidean  curvature of  $\cgam$  at $p$ equals the curvature of its osculating circle  \footnote{The osculating circle to $\cgam$ at $p$ passes through $p$, and the derivatives of the arc-length parameterizations at $s=0$ (whith $s=0$ corresponding to $p$) of the osculating circle and $\cgam$  coincide up to the second order.}  at $p$.

Since $|N(s)|=1$, then $N_s(s)$ is proportional  to $T(s)$. Furthermore, differentiating the scalar product identity $T(s)\cdot N(s)=0$, we conclude: 
 \beq\label{e-ns} N_s(s)=-\kappa(s) T(s).\eeq
 Equations \eqref{e-ts} and    \eqref{e-ns} are called \emph{Frenet equations} and can be written in the matrix form as
 $$A_s=C A,$$
 where $A$ is the Euclidean frame matrix \eqref{A}, while 
 \beq\label{cartan-m} C(s) =A_s(s)A(s)^{-1}=
\begin{pmatrix}
    0&\kappa(s) \\
  -\kappa(s)&0
\end{pmatrix}
\eeq
is the Euclidean \emph{Cartan matrix}.  From  the equivariance property \eqref{frame-equiv} and  the \emph{$SE(2)$-invariance}\footnote{The Euclidean curvature $\kappa$ changes its sign under reflections and, therefore, is not invariant under the full Euclidean group $E(2)$. Nonetheless, it is customary called the \emph{Euclidean curvature} rather than the \emph{special Euclidean curvature}.}  of $C$ (and, therefore, of $\kappa$) it  follows:
$$\kappa_{g\cgam}(gp)=\kappa_{\cgam}(p),$$
where $\cgam$ is an arbitrary curve, $p\in \cgam$,  $g\in SE(2)$.  
\begin{figure}
\begin{minipage}{.45\linewidth}
\centering
  \includegraphics[width=7cm]{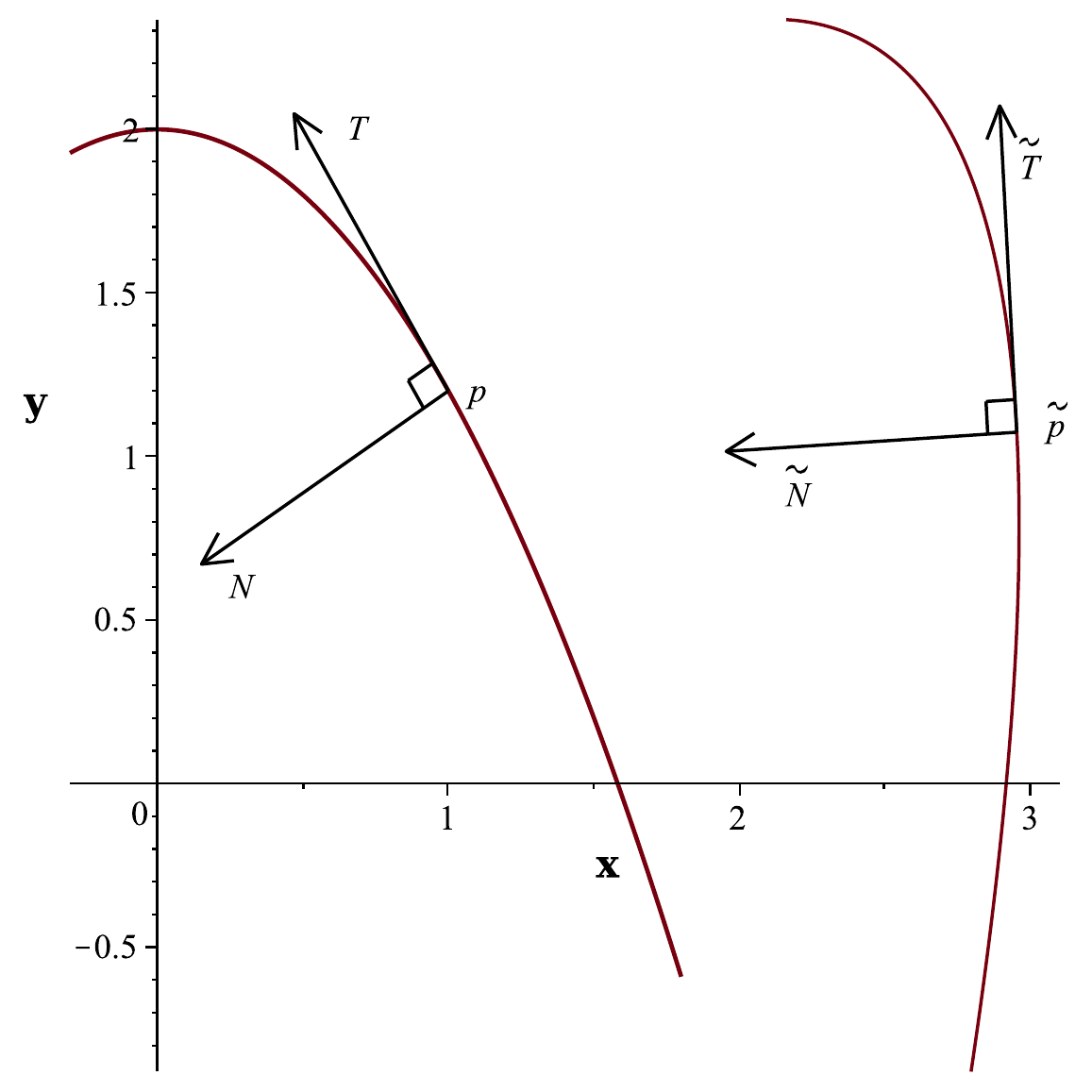}
  \caption{The $SE(2)$-action preserves the lengths of vectors and the angle between them.}
  \label{fig:eucl.mf}

\end{minipage}
\hspace{.05\linewidth}
\begin{minipage}{.45\linewidth}

  \centering
  \includegraphics[width=7cm]{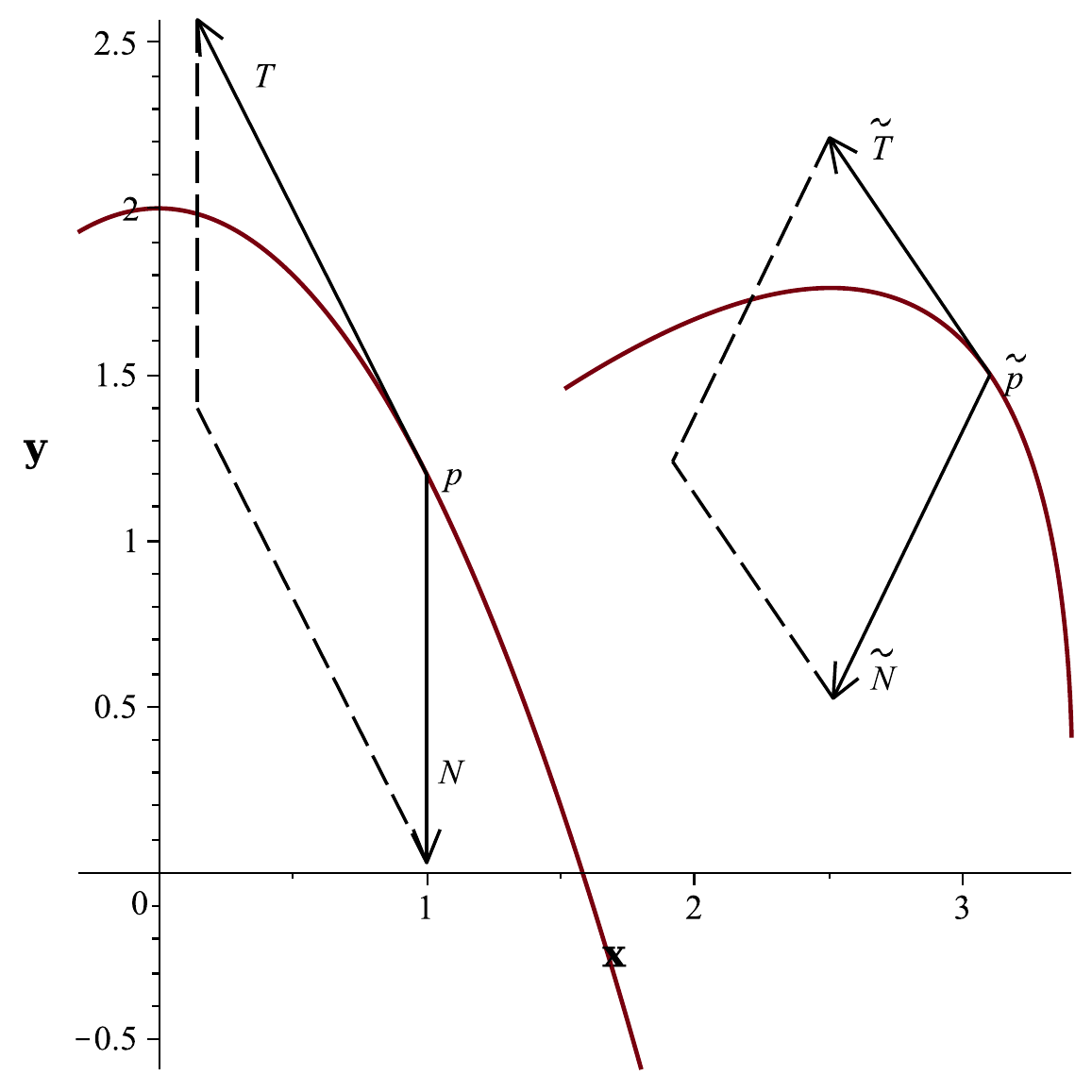}
  \caption{The $SA(2)$-action preserves the area of the parallelogram defined by the affine tangent and normal vectors, but not  their lengths or the angle between them.}
  \label{fig:aff.mf}

\end{minipage}
\end{figure}
\subsection{Affine moving frame and invariants}  \label{ssect-amf}
The action of the special affine group $SA(2)$ preserves neither Euclidean distances nor angles. Thus the Euclidean moving frame consisting of the unit tangent and  the unit normal at each point of a curve $\cgam$ is not compatible  with the   $SA(2)$-action.   However, the $SA(2)$-action preserves areas, and we can use this property to define an $SA(2)$-equivariant frame.

 It turns out that any $C^2$-smooth curve $\cgam$ can be parametrized by 
 $$\gamma\colon \al\to p=\gamma(\alpha),$$ so that  the area of the parallelogram defined by  vectors 
 \beq\label{aTN}T(p)=\gamma_\al \text{ and } N(p)=\gamma_{\al\al}\eeq
 is 1 and the closest rotation from $T(p)$ to $N(p)$ is counter-clockwise.  The parameter $\al$ is called the \emph{affine arc-length parameter}. Explicitly, if 
 $${ \hat \gamma\colon t\to\hat\gamma(t)=(x(t),y(t))}$$ is any parametrization of $\cgam$, then
\beq\label{alpha}\al(t)=\int_0^t\det(\hat\gamma_\tau(\tau),\hat\gamma_{\tau\tau}(\tau))^{1/3}d\tau.\eeq
Recalling formulas \eqref{ds} and  \eqref{kappa}, we rewrite \eqref{alpha} in terms of the Euclidean curvature and arc-length:
\beq\label{alpha2}\al(s)=\int_0^s\kappa(\tau)^{1/3}d\tau.\eeq
 Vectors $T(p)=\gamma_\al$ and $N(p)=\gamma_{\al\al}$ are called the \emph{affine tangent and normal} to ${\cgam}$ at $p$, respectively. It is important to note that  although $T(p)$ is tangent to ${\cgam}$ at $p$, it is, in general, not of the unit length,  while   $N(p)$, in general, is neither perpendicular to  $T(p)$  nor of the unit length. The \emph{$SA(2)$-frame matrix} is then defined by
\beq \label{Aa}A_{\cgam}(p)=
\begin{pmatrix}
    T(p) \\
  N(p)  
\end{pmatrix}=\begin{pmatrix}
    \gamma_\al \\
\gamma_{\al\al}  \\
\end{pmatrix}.
\eeq
An important observation is that, by construction,  $\det(A_{\cgam}(p))=1$. In fact,  this is  the matrix of the unimodular linear  transformation
 which brings the affine  moving frame basis consisting of $T(p)$ and  $N(p)$ to  the standard orthonormal basis under the action on row vectors $\v\to \v M^{-1}$. 
 
 The affine moving frame is $SA(2)$-equivariant: an element $g\in SA(2)$ mapping the curve $\cgam$ to $\tilde\cgam$ and the point $p\in\cgam$ to the point $\tilde p\in\tilde\cgam$, also maps the affine  tangent and normal { vectors} at $ p\in \cgam$  to the affine tangent and normal vectors { at} $\tilde p\in\tilde\cgam$.  See Figure~\ref{fig:aff.mf} for an illustration. In the matrix form, this can be expressed as 
\beq\label{frame-equiv-A}A_{g\,\cgam}(g p)=A_{\cgam}(p)M_g^{-1},\eeq
where $\cgam$ is an arbitrary curve, $p\in \cgam$,  $g\in SA(2)$, and $M_g$ is the matrix part of $g$.

By definition:
\beq\label{ta} T_{\al}(\al)=N(\al).\eeq
Using this and differentiating  the identity $\det(T(\al), N(\al))=1$ with respect to $\al$ we obtain $\det(T(\al), N_\al(\al))=0$.
This implies that $N_\al$ is proportional to $T$, and, therefore,  there is a function $\mu(\al)$, called the \emph{affine curvature function}, such that
 \beq\label{na}N_\al(\al)=-\mu(\al) T(\al),\eeq 
where  
\beq\label{eq-aa} \mu(\al)=-\det(N_\al(\al), N(\al))=\det(\gamma_{\al\al}(\al), \gamma_{\al\al\al}(\al)).\eeq 
 If $\hat\gamma(t)$ is an arbitrary parameterization of $\cgam$, then the formula for $\mu(t)$ is rather long (see formula (7-24) in \cite{gugg}), but we can get a more concise formula in terms of the Euclidean  curvature and the Euclidean arc-length \cite{kogan2003}:
\beq\mu=\frac{3 \kappa (\kappa_{ss}+3 \kappa^3)-5 \kappa_s^2}{9 \kappa^{8/3}}.\eeq
The  affine  curvature of a conic is constant (see Section~\ref{sect-arecon} for the details).  
The  affine  curvature of  $\cgam$  at $p$ is the curvature of the osculating conic \footnote{The osculating conic to $\cgam$ at $p$ passes through $p$, and the derivatives of the affine arc-length parameterizations at $\al=0$ (with $\al=0$ corresponding to $p$) of the osculating conic and $\cgam$  coincide up to the third order.} at $p$.

Equations \eqref{ta} and    \eqref{na} are the affine version of the {Frenet equations} and can be written in the matrix form as
 \beq\label{eq-Aa}A_\al(\al)=C(\al) A(\al),\eeq
 where $A$ is the affine frame matrix \eqref{Aa}, while 
 \beq\label{cartan-ma} C(\al) =A_\al(\al)A(\al)^{-1}=
\begin{pmatrix}
    0&1 \\
  -\mu(\al)&0
\end{pmatrix}
\eeq
is the  \emph{affine Cartan matrix}. { From  the equivariance property \eqref{frame-equiv-A} and  the \emph{$SA(2)$-invariance}\footnote{The affine curvature $\mu$ is scaled under non-unimodular linear transformations and, therefore, is not invariant under the full affine group $A(2)$. Nonetheless, following \cite{gugg}, we use the term \emph{affine curvature} rather than the \emph{special} or \emph{equi-}affine curvature.}  of $C$ (and, therefore, of $\mu$) it follows}:
$$\mu_{g\cgam}(gp)=\mu_{\cgam}(p),$$
where $\cgam$ is an arbitrary curve, $p\in \cgam$, and $g\in SA(2)$. 

\subsection{Norms and distances} \label{ssect-norms}

 For a continuous function $f(t)$ on a closed interval $[0,L]$, let
 \beq\label{infinf} \infin{f}:=\max_{t\in[0,L]}\{|f(t)|\}.\eeq
  For a $k\times\ell$ matrix $A$ \emph{with real entries}   we define:
  \beq\label{def-maxc}\maxc{A}:= \max_{{i=1,\dots,k}\atop {j=1,\dots,\ell}}\{|a_{ij}|\},\eeq
   where $a_{ij}$ are  the entries of $A$ and $|\cdot|$ is the usual absolute value. If   $A(t)$ is a matrix  whose entries are  functions on a real interval $ [0,L]$, we define a real valued function   
\beq\label{maxc}\maxc{A}(t):=\maxc{A(t)}.\eeq
If the entries of $A(t)$ are continuous functions, it is easy to show that  $\maxc{A}(t)$ is continuous on the interval $[0,L]$ and so  we may define: 
\beq\label{infinA}\infin{A}:=\infin{\maxc{A}(t)}=\max_{t\in[0,L]}\maxc{A(t)}=\max_{{t\in[0,L]}\atop{{i=1,\dots,k}\atop {j=1,\dots,\ell}}}\{a_{ij}(t)\},\eeq
where the first equality is  the definition, and the subsequent equalities follow from  \eqref{infinf}--\eqref{maxc}.
 
We note that  $\maxc{\cdot}$ and $\infin{\cdot}$ are $L^\infty$-norms on  the vector spaces of matrices of matching sizes with real entries and  functional entries, respectively, and, in particular, they satisfy the triangle inequality.

As usual, the  differentiation and integration of matrices with functional entries  are defined component-wise.  For a matrix   $A(t)$,   whose entries are continuous functions on a real interval $ [0,L]$, and $t\in [0,L]$ we will repeatedly use the inequalities:
 \beq\label{ineq}\maxc{\int_0^t A(\tau)d\tau}\leq  \int_0^t \maxc{A}(\tau)d\tau\leq  ||A||_{[0,t]}t\leq  ||A||_{[0,L]}t \leq  ||A||_{[0,L]} L.\eeq

For a vector $\v\in \R^\ell$, its  $L^\infty$-norm $\maxc{\v}$ and its  Euclidean $L^2$-norm $|\v|$ obey the following inequality:
\beq \label{eq-inf-2} |\v|\leq\sqrt{\ell} \maxc{\v}.\eeq
 In this paper, the closeness of two curves is determined by the Hausdorff distance, and we  recall its definition. Let $P$ and $Q$ be two subsets of $\R^n$. We define
$$ { d_{PQ}}=\sup_{p\in P}\inf_{q\in Q}|p-q| \text{ and } { d_{QP}}=\sup_{q\in Q}\inf_{p\in P}|p-q|.$$
Then the \emph{Hausdorff} distance between $P$ and $Q$ is defined by
$$d(P,Q)=\max\{ d_{PQ}, d_{QP}\}.$$
To  find an upper bound for the Hausdorff distance between two \emph{planar curves} $\cgam_1$ and $\cgam_2$ parameterized by  $\gamma_1(t)$ and $\gamma_2(t)$ for $t\in [0,L]$  we note that
\begin{align} 
\nonumber  { d_{\cgam_1\,\cgam_2}}&=\sup_{\tau\in [0,L]}\inf_{t\in [0,L]}|\gam_1(\tau)-\gam_2(t)| \leq \sup_{\tau\in [0,L]}|\gam_1(\tau)-\gam_2(\tau)|\leq \sqrt 2 \sup_{\tau\in [0,L]} \maxc{\gamma_1(\tau)-\gamma_2(\tau)}\\
\nonumber 
&=\sqrt 2\infin{\gamma_1-\gamma_2}.
\end{align}
The same inequality holds for  $ { d_{\cgam_2\,\cgam_1}}$ and, therefore, for the Hausdorff distance we have:
\beq \label{hcg}d(\cgam_1,\cgam_2)\leq\sqrt 2\infin{\gamma_1-\gamma_2}.\eeq
\subsection{Convergence}\label{ssect-conv}
We recall the definition of uniform convergence:
\begin{definition} Let $\{f_n\}_{n=1}^\infty$ be a sequence of real valued functions on   a set $P$. We say that $\seq{f_n}$ converges  to a function $f$ \emph{uniformly}  on $P$ if  for every $\ep>0$, there exists $n_\ep$, such that
\beq\label{eq-uconv} |f_n(p)-f(p)|< \ep\quad \quad \text{ for all $n>n_\ep$ and all $p\in P$. }\eeq
\end{definition}
The difference between the uniform and \emph{point-wise} convergence is that  one can choose $n_\ep$ which ``works'' for all $p\in P$. If $P$ is an interval $[0,L]$, then uniform convergence of $\{f_n\}$ to $f$ is equivalent to
$$\limn\infin{f_n-f}=0.$$ 

\begin{lemma}\label{lem-uconv}  Let  $\{f_n\}_{n=1}^\infty$ be a sequence of real valued functions on   a domain $P$ \emph{uniformly} convergent  to a function $f$ on $P$.  Assume further that each of the functions $f_n$, and also $f$ { achieves, its maximum value} on $P$, then
\beq\label{eq-uconv-max}\lim_{n\to\infty}\max_{p\in P}\{f_n(p)\}=\max_{p\in P}\{f(p)\}.\eeq
\end{lemma}
\begin{proof}
By assumption there exist  $\{p_n\}\subset P$,  $n=0,\dots,\infty$, such that for all  $p\in P$:
$$ f(p)\leq f(p_0)=m_0 \quad\text {and}\quad   f_n(p)\leq f_n(p_n)=m_n,   \quad n\in \Z_+,$$
where $m_0$ is the maximal value of $f$ and $m_n$ is the maximal value of $f_n$, $ n\in \Z_+$, on $P$.  
Identity \eqref{eq-uconv-max} can be rewritten as
\beq\label{eq-epm}\lim_{n\to\infty}m_n=m_0.\eeq
For an arbitrary $\ep>0$, let $n_\ep$ be such that for all $n>n_\ep$ and all $p\in P$ \eqref{eq-uconv} holds, and so
for all $n>n_\ep$ and all $p\in P$:
\beq\label{eq-uconv1} f(p)-\ep <f_n(p)<f(p)+\ep.\eeq
Substitute $p_0$ in the left inequality in \eqref{eq-uconv1} to get
\beq\label{eq-uconv2} f(p_0)-\ep =m_0-\ep<f_n(p_0)\leq m_n.\eeq
Substitute $p_n$ in the right inequality in \eqref{eq-uconv1} to get
\beq\label{eq-uconv3} f_n(p_n)=m_n<f(p_n)+\ep\leq m_0+\ep.\eeq
Together \eqref{eq-uconv2} and \eqref{eq-uconv3} imply that for  an arbitrary $\ep>0$,  there exists  $n_\ep$ such that for all $n>n_\ep$
\beq\nonumber m_0-\ep <m_n<m_0+\ep,\eeq
which is equivalent to \eqref{eq-epm}.
\end{proof}
We say that a sequence of $k\times\ell$ matrices $\left\{A_n\right\}_{n=1}^\infty$ with real entries $a_{n;ij}$ converges to a  $k\times\ell$ matrix  $A$ with real entries $a_{ij}$, if for all $i=1,\dots, k$, $j=1,\dots,\ell$:  $$\limn a_{n;ij}=a_{ij}.$$ If $\left\{A_n(t)\right\}$ is a sequence of matrices whose elements are  real valued functions on an interval $[0,L]$, then we say that  $\left\{A_n(t)\right\}_{n=1}^\infty$ point-wise converges to  $A(t)$ if for all $t\in [0,L]$ and all $i=1,\dots, k$, $j=1,\dots,\ell$:
$$\limn a_{n;ij}(t)=a_{ij}(t).$$ If the latter convergences are uniform  on $[0,L]$, we say that $\left\{A_n(t)\right\}$ converges to  $A(t)$ uniformly. Equivalently, the uniform convergence can be defined by 
 $$\limn ||A_n- A||_{[0,L]}=0.$$

From  Lemma~\ref{lem-uconv}, we have the following important corollary, which we use repeatedly.
\begin{corollary}\label{cor-lim} \hfill
\begin{enumerate}
\item Let $\seq{A_n}_{n=1}^\infty$ be a sequence of matrices with real entries convergent  to a matrix $A$, then 
\begin{align}\label{eq-lim1} \limn \maxc{A_n}&= \maxc{A}.\end{align}
\item Let $\seq{A_n(t)}_{n=1}^\infty$ be a sequence of matrices whose elements are real valued functions on the interval $[0,L]$   point-wise convergent to a matrix of functions $A(t)$. Then for all $t$
\begin{align}\label{eq-lim2} \limn \maxc{A_n}(t)&= \maxc{A}(t).\end{align}
\item If  the entries of $A_n(t)$ are continuous functions and $\{A_n(t)\}_{n=1}^\infty$ converges  to $A(t)$ uniformly on $[0,L]$, then
\begin{align}\label{eq-lim3} \limn \infin{A_n}&= \infin{A}.\end{align}
\end{enumerate}
\end {corollary}
\begin{proof}
\begin{enumerate}
\item Identity  \eqref{eq-lim1} is equivalent to 
\beq\label{eq-pf-lim1} \lim_{n\to\infty}\max_{{i=1,\dots,k}\atop {j=1,\dots,\ell}}\{|a_{n, ij}|\}  =\max_{{i=1,\dots,k}\atop {j=1,\dots,\ell}}\{|\lim_{n\to\infty}a_{n, ij}|\}.\eeq

Let  $B_n$, $n \in \mathbb{Z}_+$, and $B$ denote matrices whose elements are $|a_{n, ij}|$ and $|a_{ij}|$, respectively. Then, due to a well known and easy to show fact that $\lim$ and the absolute value are interchangeable, $\limn B_n=B$.  Note that a $k\times\ell$ matrix  with real entries can be viewed as a real valued function on a finite set of ordered pairs   
\beq\label{setP} P=\{ (i,j)|i=1,\dots, k, j=1,\dots,\ell\}.\eeq 

Viewed as  a sequence of such functions, $\seq{B_n}_{n=1}^\infty$ converges to $B$ uniformly on $P$.  Any function on a finite set attains its maximum and  so we can apply Lemma~\ref{lem-uconv} to conclude that
 $$ \limn\max_{p\in P} \{B_n(p)\} =\max_{p\in P}  \limn \{B_n(p)\},$$ which  is equivalent to  \eqref{eq-pf-lim1}.
 
\item  Identity  \eqref{eq-lim2} is an immediate consequence of \eqref{eq-lim1}.
 
\item Identity  \eqref{eq-lim3} is equivalent to 
\beq\label{eq-pf-lim3} \lim_{n\to\infty}\max_{{t\in[0,L]}\atop{{i=1,\dots,k}\atop {j=1,\dots,\ell}}}\{|a_{n, ij}(t)|\}  =\max_{{t\in[0,L]}\atop{{i=1,\dots,k}\atop {j=1,\dots,\ell}}}\{|\lim_{n\to\infty}a_{n, ij}(t)|\}.\eeq 
 Let $B_n(t)$ and $B(t)$ denote matrices whose elements are $|a_{n, ij}(t)|$ and $|a_{ij}(t)|$, respectively. Then $\{B_n(t)\}$ converges to $B(t)$ uniformly on $[0,L]$. Uniform convergence implies that entries of $B(t)$ are continuous.   We can view  a $k\times\ell$ matrix  whose entries are continuous functions on $[0,L]$ as real valued functions on the set
  $$Q=P\times [0,L],$$
  where $P$ is defined by \eqref{setP}.
With this point of view, the sequence of functions  $\seq{B_n}_{n=1}^\infty$ converges to $B$ uniformly on $Q$, and  each of these functions attains its maximum value on $Q$. Thus  they satisfy the assumptions of Lemma~\ref{lem-uconv}, and so 
 $$ \limn\max_{q\in Q} \{B_n(q)\} =\max_{q\in Q}  \limn \{B_n(q)\},$$
 which is equivalent to \eqref{eq-pf-lim3}.

 \end{enumerate}
\end{proof}

\section{Euclidean reconstruction}\label{sect-erecon}
In this section, we review how a  curve  can be reconstructed  from its Euclidean curvature by  two successive integrations (Theorem~{thm-euc-rec}). We then use these formulas to estimate how close, relative to the Hausdorff  distance, two curves can be brought together by a special-Euclidean transformation, provided their Euclidean curvatures as functions of the Euclidean arc-length are $\delta$-close in the $L^{\infty}$-norm (Theorem~\ref{thm-euc-est}) or  $\delta$-close in the $L^{1}$-norm (Theorem~\ref{l1-norm}) .  

\begin{theorem}[Euclidean reconstruction]\label{thm-euc-rec} Let $\ka(s)$ be a continuous function on an interval $[0,L]$. Then there is a unique, up to a special Euclidean transformation, curve $\cgam$ with the Euclidean arc-length parametrization  $\gamma(s)=\left(x(s),y(s)\right)$, $s\in [0,L]$, such that $\ka(s)=x'(s)y''(s)-y'(s)x''(s)$ is its Euclidean curvature.
\end{theorem}
\begin{proof}

According to \eqref{e-gs}, \eqref{e-ts}, and \eqref{e-ns},   $\gamma$ is a solution of the following system of  first order differential equations:
\begin{align}
\label{gamma'Te} \gamma'(s)&=T(s)\\
\label{T'Ne} T'(s)&=\ka(s) N(s)\\
\label{N'Te} N'(s)&=-\ka(s) T(s),
\end{align}
Due to well known results on the existence and uniqueness of solutions to linear ODEs \cite{Nagle}, there exists a unique solution of  \eqref{gamma'Te}-\eqref{N'Te} with  initial data 
\beq\label{in-data-e} \gamma(0)=(0,0),\quad  T(0)=(1,0), \quad N(0)=(0,1).\eeq
It is easy to verify that such solution is given by
\begin{align} \label{eq-euc-rec}\gamma_0(s)=\left(\int_0^s \cos\left( \theta(t)\right)\,dt ,\int_0^s\sin(\theta\left(t)\right) dt\right),
\end{align}
where
\begin{align}\label{eq-ka-rec}\theta(t)=\int_0^s\kappa(t)dt
\end{align}
is the tangential angle, i.e.~the angle  between $T=\gam_0'(s)=\left(\cos\left( \theta(s)\right),\,\sin(\theta\left(s)\right)\right)$ and a horizontal line.
Denote a  curve  parametrized by $\gam_0$ as $\cgam_0$, and  let $\cgam_1$ be another curve with Euclidean arc-length parametrization  $\gamma_1(s)$, $s \in [0,L]$, such that $\kappa(s)$ is its Euclidean curvature.  Let $T_1(0)=\gamma_1'(0)$ and $N_1(0)=\gamma_1''(0)$. Then there exists a unique special Euclidean transformation $g\in SE(2)$ which is a composition of a translation by the vector $-\gamma_1(0)$, followed by the rotation  $\begin{pmatrix} T_1(0)\\N_1(0)\end{pmatrix}^{-1}$, such that
$$g\cdot\gamma_1(0)=(0,0), \quad  g\cdot T_1=(1,0), \quad g\cdot N_1=(0,1).$$
Since $\ka$ and $ds$  are invariant under rigid motions, it follows that the curve $g\, \cgam_1$ parametrized by $g\gamma_1$   satisfies \eqref{gamma'Te}-\eqref{N'Te} with the same initial   data \eqref{in-data-e} and, therefore, $\cgam_0= g\,\cgam_1$.
\end{proof}

Formulas \eqref{eq-euc-rec}-\eqref{eq-ka-rec} allow us to construct a curve with prescribed Euclidean curvature. The following lemma gives  a sufficient condition for a reconstructed curve to be closed.  See  Lemma 4 in \cite{Musso2009} and    Lemmas 1 and 2 in  \cite{geiger-kogan2021}.

\begin{lemma}\label{lem-closed}
 
  Let $\kappa : \mathbb{R} \to \mathbb{R}$ be a periodic continuous  function with minimum period $\ell$, 
  if 
 \beq \label{eq-closed}\frac{1}{2\pi}\int_0^\ell \kappa(s)ds = \frac{\xi}{m},\eeq
  where $m$ and $\xi$ are two relatively prime integers and $m>1$,
  then, the corresponding unit speed parameterization $\gamma$, given by \eqref{eq-euc-rec},  defines a closed curve. The map $\gamma$ has minimal period $m\ell$.   The turning number of $\gamma$ over the interval $[0, m\ell]$ is equal to $\xi$. If $\cgam=Im(\gamma)$ is simple, then  $\xi=1$ and  $m$ is the $SE(2)$-symmetry index of $\cgam$.
\end{lemma}

\begin{ex}
  To illustrate the above lemma, consider the function
  \beq \label{eq-musso-ex1} \kappa_1(s) = \sin(s) + \cos(s) + \frac13. \eeq
 Then $\frac{1}{2\pi}\int_0^{2\pi} \kappa_1(s)ds = \frac13$ and the above lemma asserts that a curve with curvature function $\kappa_1(s)$ is closed  with the $SE(3)$-symmetry index of 3. Such curve, reconstructed  using  \eqref{eq-euc-rec}, is pictured in Figure~\ref{fig-mn-ex}.

  On the other hand, consider
  \beq \label{eq-musso-ex2} \kappa_2(s) = \sin(s) + \cos(s) + 1. \eeq
 Then  $\frac{1}{2\pi}\int_0^{2\pi} \kappa_2(s)ds = 1$ and the assumption $m>1$ in Lemma~\ref{lem-closed} is not satisfied. Thus the lemma does not assert that a curve for which $\kappa_2(s)$ is the Euclidean curvature function is closed. In fact, the curve reconstructed using  \eqref{eq-euc-rec} is not closed, as we can see in Figure~\ref{fig-mn-ex-1}.

  \begin{figure}
    \centering
    \subcaptionbox{A curve with Euclidean curvature  \eqref{eq-musso-ex1}. \label{fig-mn-ex}}
    {\includegraphics[width=6cm]{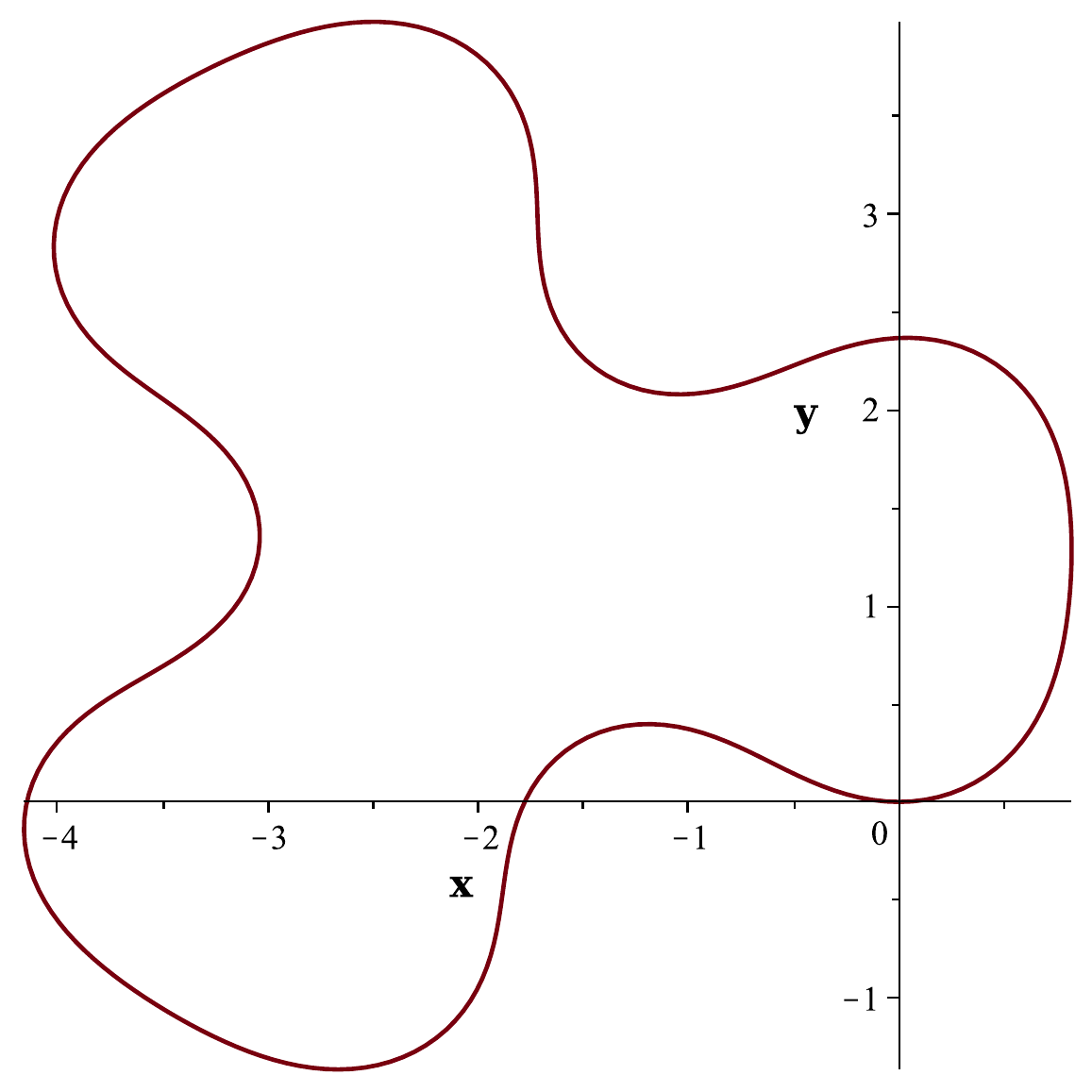}} \hspace{1cm}
    \subcaptionbox{A curve with Euclidean curvature  \eqref{eq-musso-ex2}. \label{fig-mn-ex-1}}
    {\includegraphics[width=6cm]{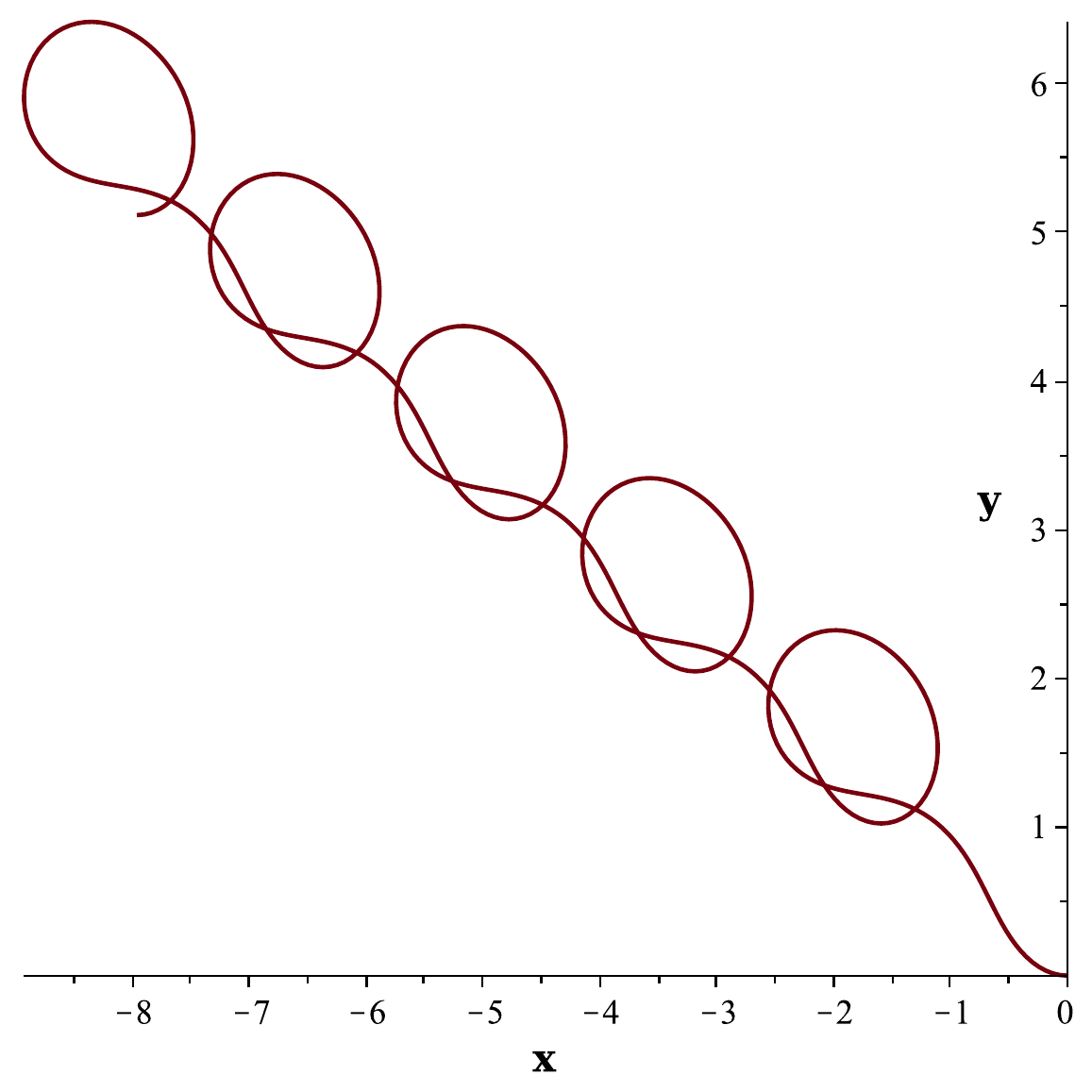}}
    \caption{Lemma~\ref{lem-closed} guarantees that the left curve is closed, but does not make any assertion about the right curve.}
  \end{figure}
  
\end{ex}

 \begin{theorem}[Euclidean estimate]\label{thm-euc-est} Let $\cgam_1$ and $\cgam_2$ be two $C^2$-smooth planar curves of the same Euclidean  arc-length $L$. Assume $\ka_1(s)$ and $\ka_2(s)$, $s\in [0,L]$ are their respective Euclidean curvature functions.  If    $||\ka_1-\ka_2||_{[0,L]}\leq \delta$, then there exists $g\in SE(2)$, such that
 \beq d(\cgam_1,g\,\cgam_2)\leq\frac{\sqrt 2} 2{\delta L^2}, \eeq
 where $d$ is the Hausdorff distance. 
 \end{theorem}
\begin{proof} Identifying $\R^2$ with $\C$ and using Euler's formula we may rewrite  \eqref{eq-euc-rec} as
\beq\label{eq-ga-comp}\gamma(s)=\int_0^se^{i\theta(t)} dt.\eeq
In what follows, we will use  an important inequality, stating that a chord is shorter than the corresponding arc, illustrated in Figure~\ref{fig-ineq-arc}: 
\beq\label{ineq-arc} \left|e^{i\theta_1}-e^{i\theta_2}\right|<|\theta_1-\theta_2|. \eeq

\begin{figure}
  \centering
  \includegraphics[width=6cm]{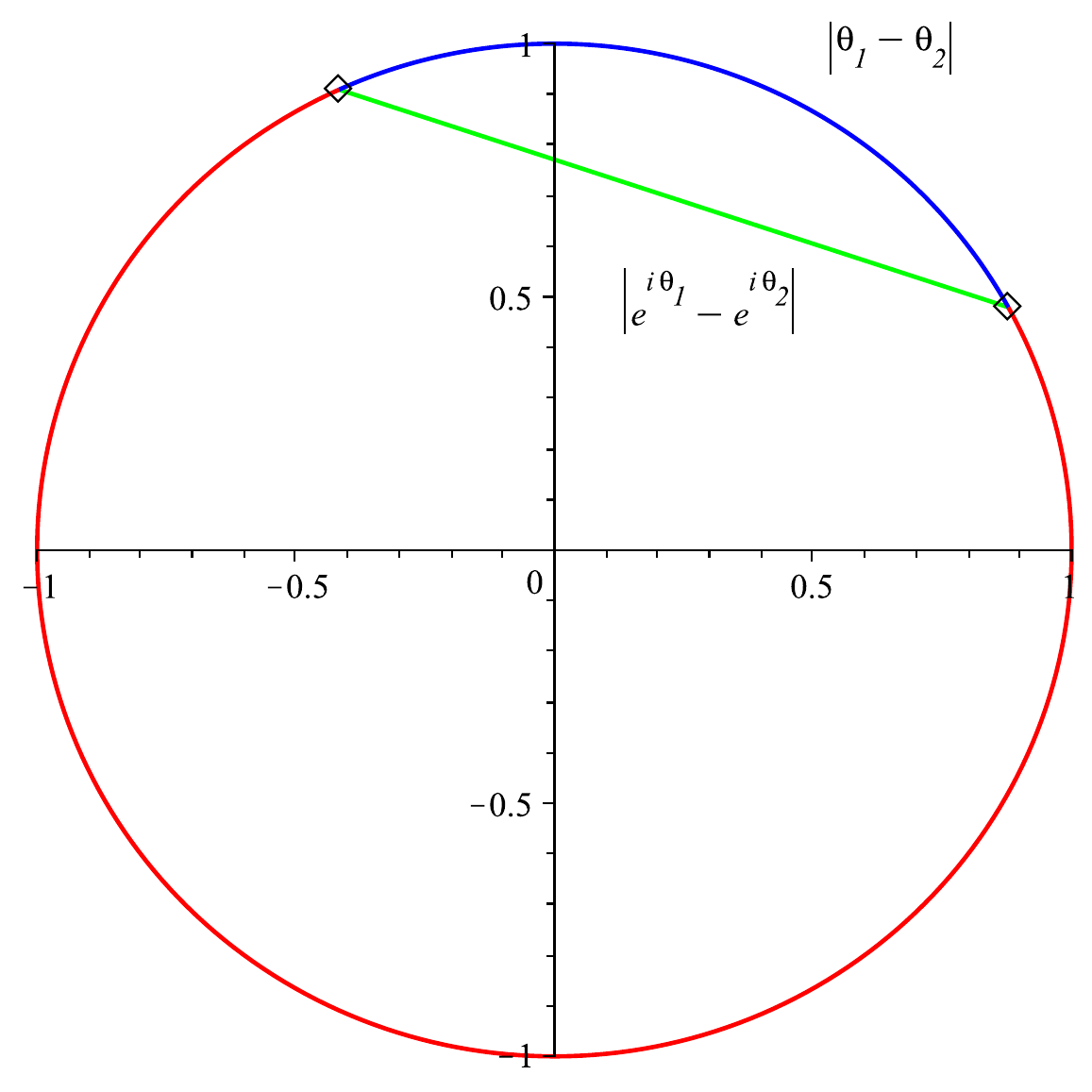}
  \caption{The length of the chord $|e^{i\theta_1}-e^{i\theta_2}|$ is shorter than the length of the arc $|\theta_1-\theta_2|$.}
  \label{fig-ineq-arc}
\end{figure}
For $j=1,2$, let $\gam_j(s)$, $s \in [0,L]$, be the Euclidean arc length parameterization of the curve $\cgam_j$. Then  $T_j(s)=\gamma_j'(s)$  and $N_j(s)=\gamma_j''(s)$ are the unit tangent and unit normal vectors, respectively, to   $\cgam_j$.  For $j=1,2$, there is a unique $g_j\in SE(2)$, such that 
 \beq\label{g-data-e} g_j\gamma_j(0)=(0,0),\quad  g_j T_j(0)=(1,0), \quad g_jN_j(0)=(0,1).\eeq
 It follows from Theorem~\ref{thm-euc-rec}, that  { $g_j\gamma_j(s)=\int_0^se^{i\theta_j(t)} dt$  for  $j=1,2$} and so:
 \begin{align} \label{eq-pf-euc-est1}  &\left| g_1\gamma_1(s)-g_2\gamma_2(s)\right|=  \left| \int_0^se^{i\theta_1(t)}dt-\int_0^s e^{i\theta_2(t)}dt\right| \leq  \int_0^s\left|e^{i\theta_1(t)}-e^{i\theta_2(t)}\right|dt\\
 \label{eq-pf-euc-est2} &< \int_0^s\left|{\theta_1(t)}-{\theta_2(t)}\right|dt= \int_0^s\left|\int_0^t\left(\ka_1(\tau)-\ka_2(\tau)\right)d\tau \right|dt
 \leq 
 \int_0^s\int_0^t\left|\ka_1(\tau)-\ka_2(\tau) \right|d\tau dt\\
 \label{eq-pf-euc-est3}
 &\leq  \int_0^s\int_0^t\infin{\ka_1-\ka_2}d\tau dt\leq \int_0^s\int_0^t \delta d\tau dt=\frac {\delta s^2} 2.
 \end{align}
 The inequality in line \eqref{eq-pf-euc-est1} follows from properties of  definite integrals, the first  inequality in line \eqref{eq-pf-euc-est2}  follows from \eqref{ineq-arc}. The equality in line \eqref{eq-pf-euc-est2}  follows from  \eqref{eq-ka-rec} and  the properties of  definite integrals. The first inequality in line \eqref{eq-pf-euc-est3} follows from \eqref{infinf}.
 
 Let $g=g_1^{-1}g_2$ then, using \eqref{hcg},   \eqref{eq-pf-euc-est1}--\eqref{eq-pf-euc-est3} and the invariance of the Euclidean distance under the rigid motions,  we have 
\begin{align*}&d(\cgam_1,g\, \cgam_2)\leq \sqrt{2} \infin{ \gamma_1-g\gamma_2} =\sqrt 2\sup_{s\in[0,L]} |\gam_1(s)-g\gam_2(s)| = 
  \sqrt 2\sup_{s\in[0,L]} |g_1\gam_1(s)-g_2 \gam_2(s)|\\
  &
  \leq \sqrt{2} \frac{\delta L^2}{2}. 
  \end{align*}
 
 \end{proof}
 If instead of the $L^\infty$-norm on the set of functions $\kappa$ we use the $L^1$-norm and require that
  $\int_0^L|\ka_1(\tau)-\ka_2(\tau)|d\tau\leq \delta$, then \eqref{eq-pf-euc-est2} implies the following result:
 \begin{theorem} \label{l1-norm} Let $\cgam_1$ and $\cgam_2$ be two $C^2$-smooth planar curves of the same Euclidean arc-length $L$. Assume $\ka_1(s)$ and $\ka_2(s)$, $s\in [0,L]$ are their respective Euclidean curvature functions and
 \beq\label{kappa-l1}\int_0^L|\ka_1(\tau)-\ka_2(\tau)|d\tau\leq \delta,\eeq
then there exists $g\in SE(2)$, such that
 $$ d(\cgam_1,g\, \cgam_2)\leq { \sqrt 2}\delta L.$$
\end{theorem}
\begin{proof}
The proof proceeds along the same  lines as the proof of Theorem~\ref{thm-euc-est}. However, the last inequality in  \eqref{eq-pf-euc-est2}  combined with \eqref{kappa-l1} implies   
$$\left| g_1\gamma_1(s)-g_2\gamma_2(s)\right|< \int_0^s \delta dt=\delta s,$$
  and so 
  $$d(\cgam_1,g\, \cgam_2)\leq \sqrt 2\sup_{s\in[0,L]} |g_1\gam_1(s)-g_2 \gam_2(s)|\leq{ \sqrt 2}\delta L.$$ 
\end{proof}
 \begin{ex}
 \begin{figure}[h!]
   \begin{minipage}{.5\linewidth}
    \centering
     {\includegraphics[width=6cm]{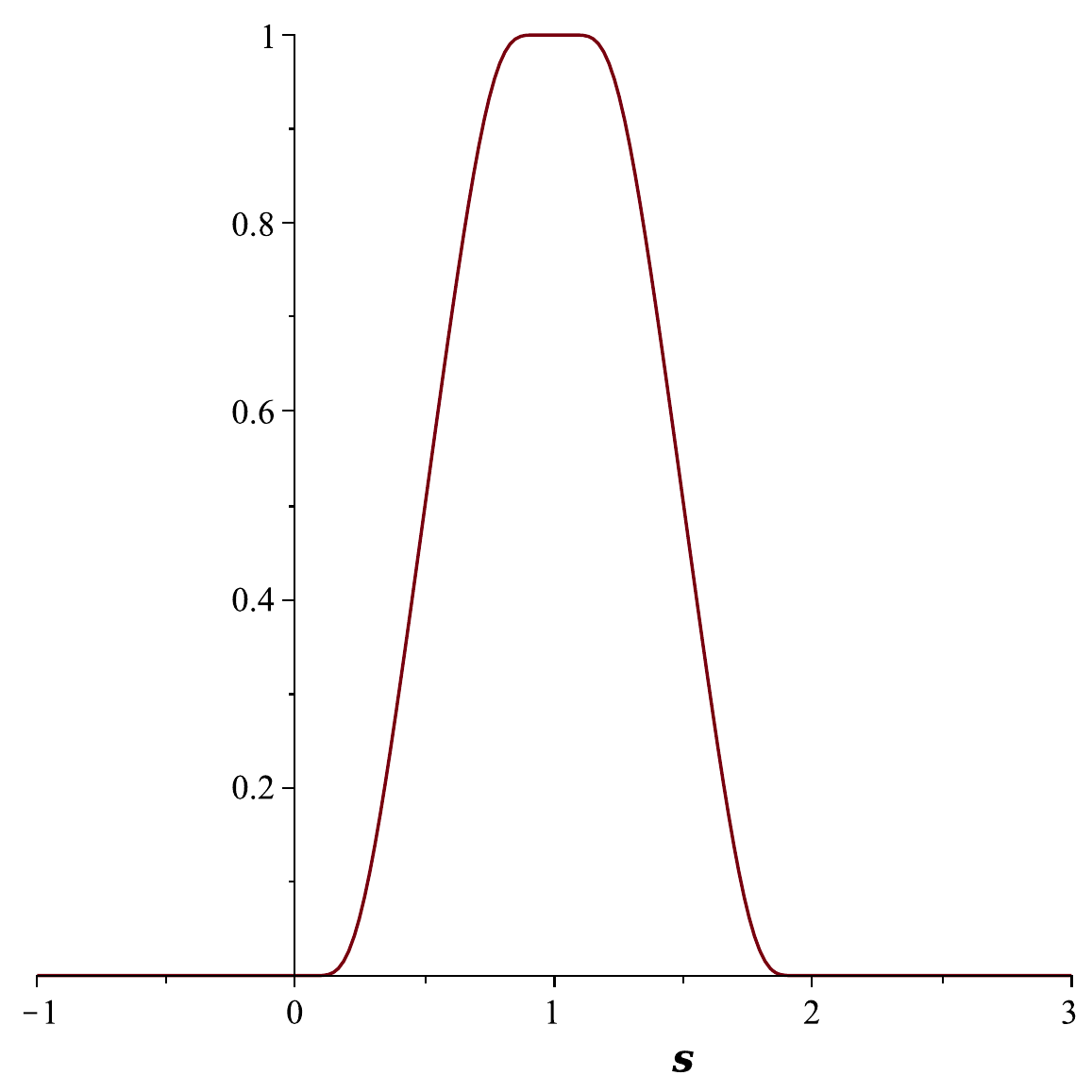}}
    \label{fig-bump}
    \caption{Bump function \eqref{eq-bump}.}

\end{minipage}
\begin{minipage}{.5\linewidth}
\centering
\includegraphics[width=6cm]{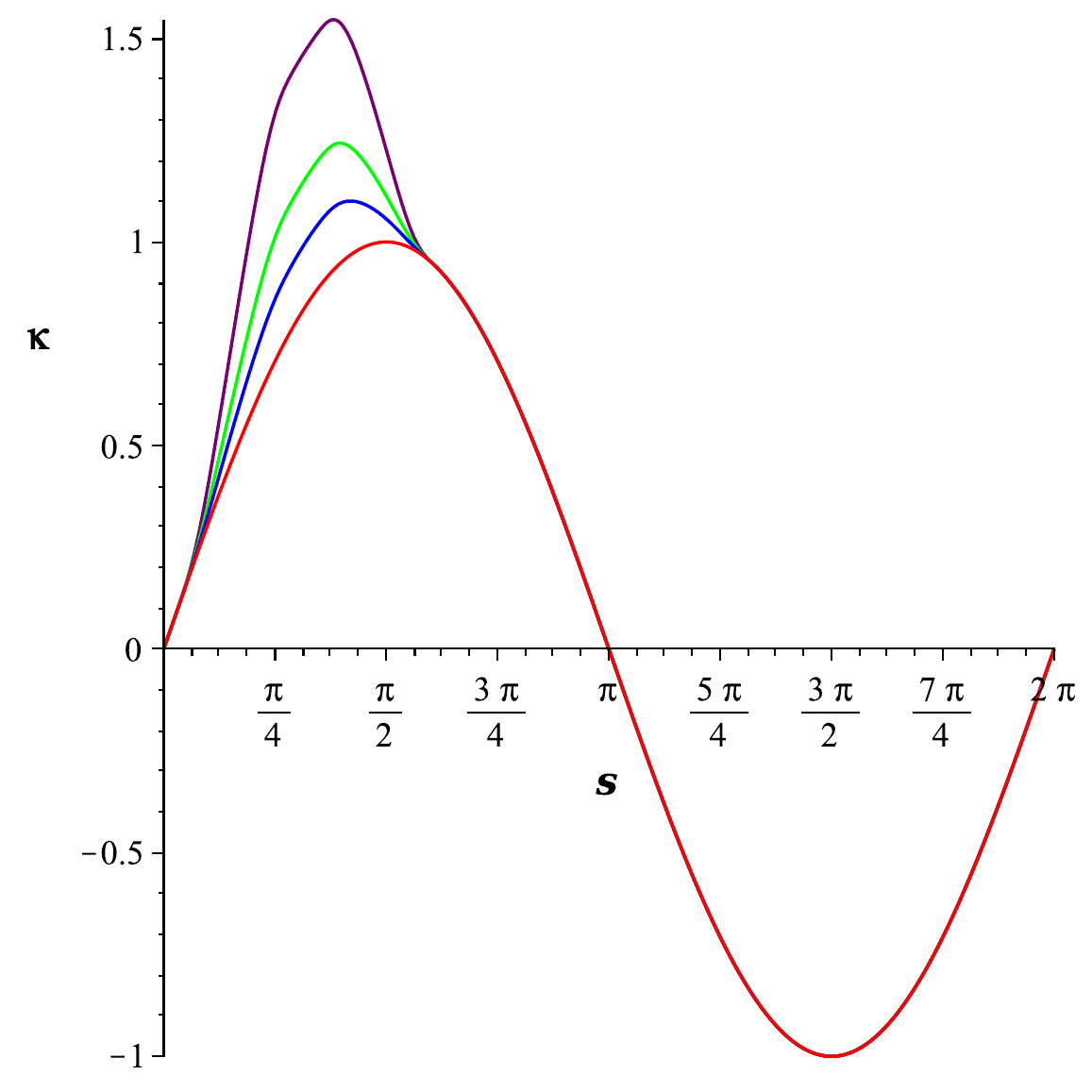}\hspace{1cm}
\caption{\textcolor{darkmagenta}{$\kappa^*_{10}(s)$}, \textcolor{olivedrab}{$\kappa^*_{20}(s)$}, \textcolor{darkblue}{$\kappa^*_{40}(s)$}, given by \eqref{eq-kn} and \textcolor{brown}{$\kappa(s)$}$=\sin(s)$, $s\in [0,2\pi]$.
}
   \label{fig-kn}
   \end{minipage}
   \end{figure}
 
\begin{figure}
\begin{minipage}{.5\linewidth}
\centering
\includegraphics[width=6cm]{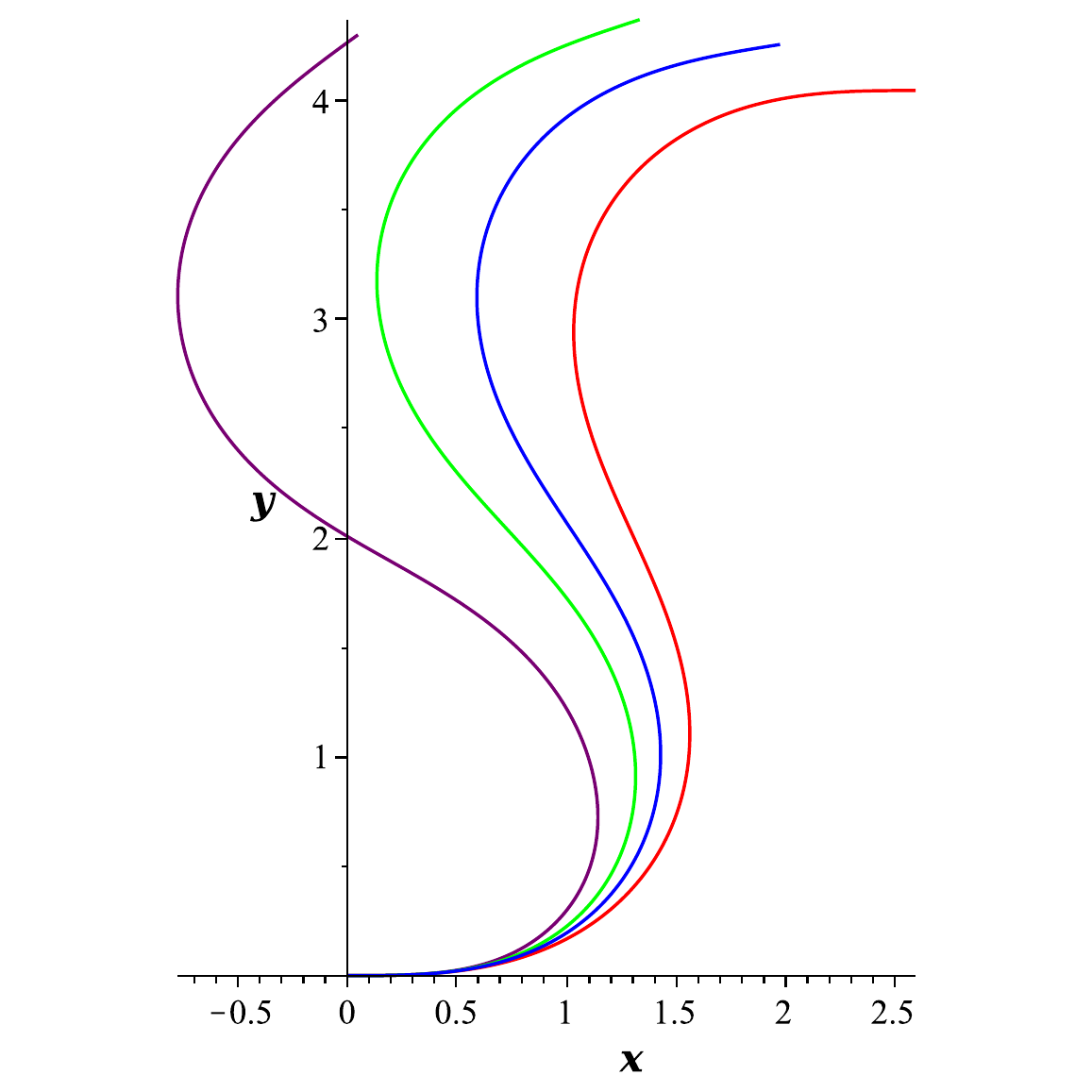} \hspace{1cm}
\caption{Curves  \textcolor{darkmagenta}{$\cgam_{10}$},   \textcolor{olivedrab}{$\cgam_{20}$}, \textcolor{darkblue}{$\cgam_{40}$},  and \textcolor{brown}{$\cgam$}, reconstructed from the Euclidean curvature functions in Figure~\ref{fig-kn}.}
    \label{recon-kn}
   \end{minipage}
\begin{minipage}{.5\linewidth}
\centering
     {\includegraphics[width=6cm]{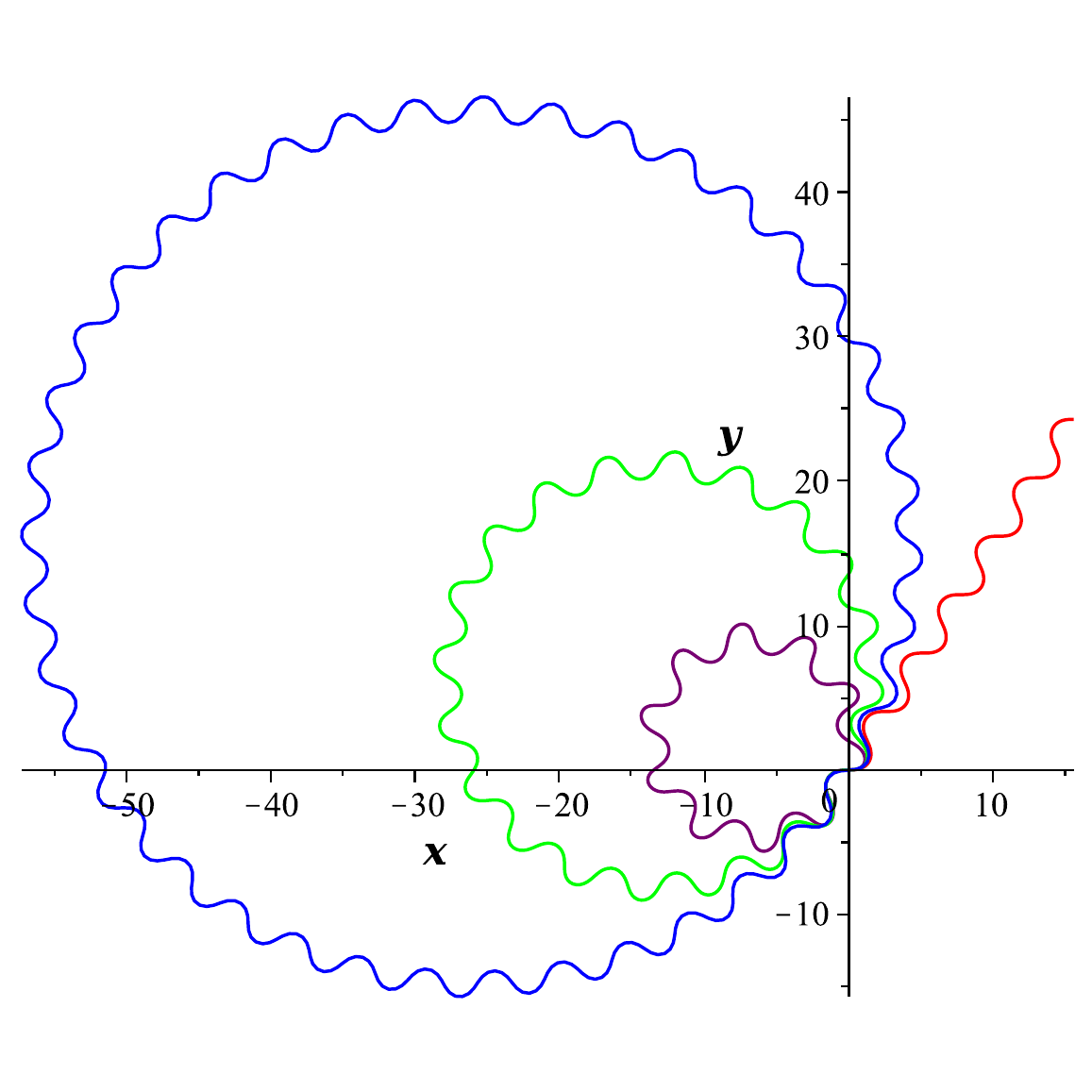}}
    \label{fig-Gamman}
    \caption{Closed curves reconstructed from  periodic extensions  \textcolor{darkmagenta}{$\kappa_{10}(s)$}$, s\in[0,20\pi]$, \textcolor{olivedrab}{$\kappa_{20}(s)$}$, s\in[0,40\pi]$,  \textcolor{darkblue}{$\kappa_{40}(s)$}$, s\in[0,80\pi]$ of $\kappa^*_n$, shown in Figure~\ref{fig-kn}, and an  open  curve reconstructed from \textcolor{brown}{$\kappa(s)$}$=\sin(s)$, $s\in [0,12\pi]$.}
    \label{fig-closed-kn}
    \end{minipage}
   \end{figure}

   To illustrate Theorems \ref{thm-euc-est} and~\ref{l1-norm}, we consider a curve whose Euclidean curvature function  is $\kappa(s)=\sin(s)$  and a family of curves   obtained by  some variations of $\kappa(s)$. To define these variations consider a  smooth bump function:
   \begin{equation}
     \label{eq-bump}
     f(s) = \begin{cases}
       0 &\text{if }   s \leq 0, \\
       \frac{e^\frac{1}{1-s}}{e^\frac{1}{s}+e^\frac{1}{1-s}} &\text{if }  0 < s < 1, \\
       1 &\text{if }  s = 1, \\
       \frac{e^\frac{1}{s-1}}{e^\frac{1}{s-1}+e^\frac{1}{2-s}} &\text{if }  1 < s < 2, \\
       0 &\text{if }  s \geq 2.
     \end{cases}
   \end{equation}

 Next, for $n\in \Z\backslash \{0\}$,  we define functions
 \begin{equation}\label{eq-kn}
   \kappa_n^*(s) = \sin(s) + \frac{2\pi}{n}f(s)
 \end{equation}
 on the closed interval $[0,2\pi]$ and let $\kappa_n(s)$ denote the periodic extension of $\kappa_n^*$ to $\mathbb{R}$.  We observe that for any $L>0$,
 $$||\ka_n-\ka||_{[0,L]}\leq\left|\left|\frac{2\pi}{n}f(s)\right|\right|_{[0,2]}\leq\frac{2\pi}{|n|}.$$
 As $|n|\to\infty$,  for $n>0$ and for $n<0$, the sequence  $\kappa_n(s)$ uniformly converges to $\sin(s)$. In Figure~\ref{fig-kn}, we show 
 $\kappa_{10}(s)$, $\kappa_{20}(s)$, $\kappa_{40}(s)$, and $\kappa(s)=\sin(s)$ over their minimal period $[0,2\pi]$, while in Figure~\ref{recon-kn}, we show  curves $\cgam_{10}$,   $\cgam_{20}$, $\cgam_{40}$,  and $\cgam$ reconstructed from these curvatures with $s\in[0,2\pi]$. We observe that the Hausdorff distance between $\cgam$ and $\cgam_n$ decreases as  $|n|$ increases (and so $\delta=\frac{2\pi}{|n|}$ decreases). At the same time, if we restrict $s$ to an interval $[0,L]$, with $0<L\leq 2 \pi$, then for a fixed $n$,  as $L$ increases,   the distance between $\cgam$ and $\cgam_n$  increases.

 Since $\int_0^{2} f(s)ds = 1$, then $\int_0^{2\pi} \kappa_n(s) ds = \frac{2\pi}{n}$. 
 Therefore, by Lemma \ref{lem-closed}, for $n \in \mathbb{Z}\setminus\{-1,0,1\}$, a curve reconstructed  from $\kappa_n(s)$ with $s\in [0, 2\pi n]$ is a closed curve with symmetry index $|n|$ and turning number 1.  A curve reconstructed from $\kappa(s)=\sin(s)$ is, however, not closed. 
 In Figure~\ref{fig-closed-kn}, we show the closed curves reconstructed from  the curvatures $\kappa_{10}(s), s\in[0,20\pi]$, $\kappa_{20}(s), s\in[0,40\pi]$,  $\kappa_{40}(s), s\in[0,80\pi]$, as well as an open  curve reconstructed from $\kappa(s)=\sin(s)$, with $s\in [0,12\pi]$.

  \begin{figure}[h!]
    \centering
    \subcaptionbox{A curve with curvature $\kappa_{5/3}(s)$, $s\in [0,10\pi]$ has turning number $3$ and  $SE(2)$-symmetry index $5$.}
    {\includegraphics[width=6cm]{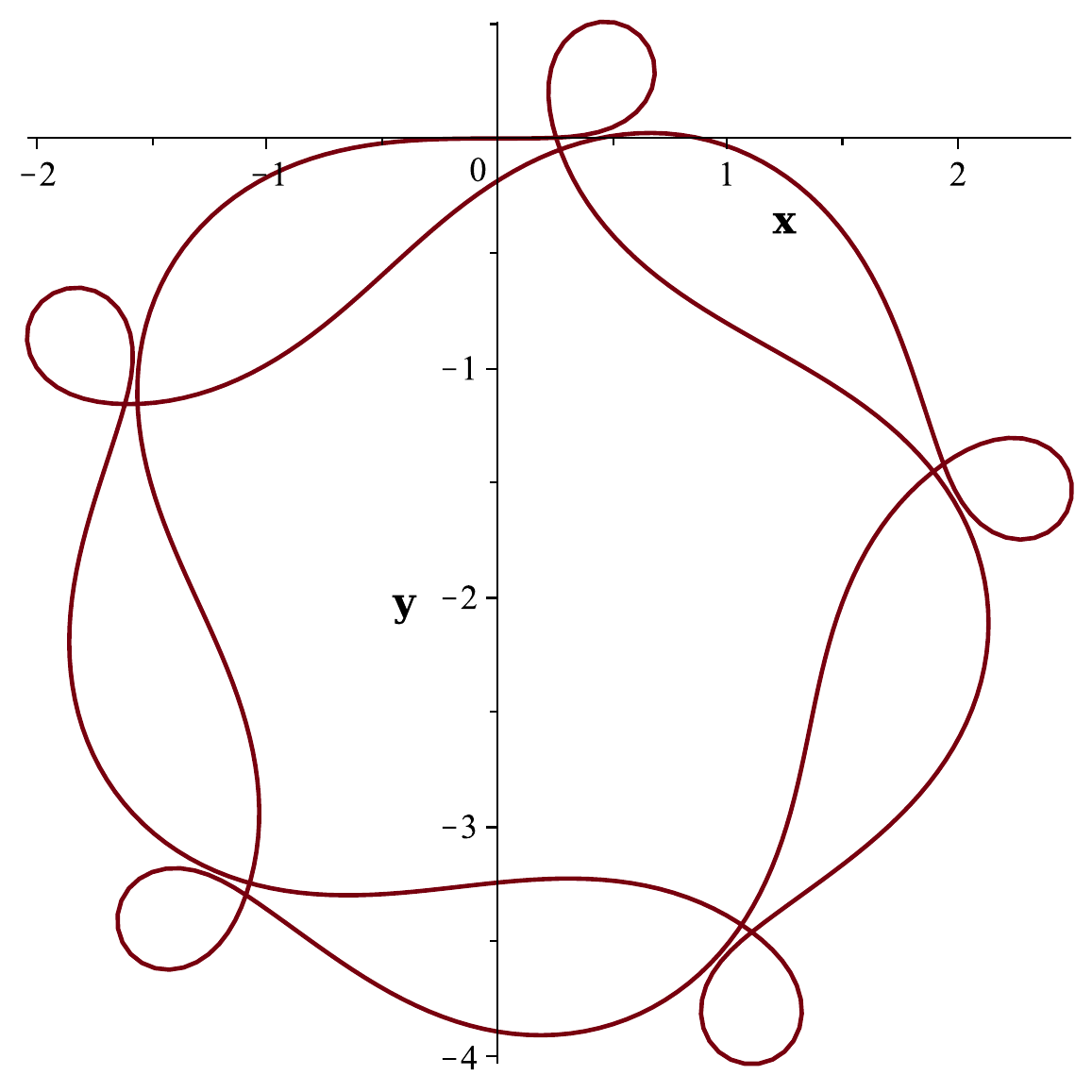}} \hspace{1cm}
    \subcaptionbox{A curve with curvature $\kappa_{3/5}$, $s\in [0,6\pi]$  has turning number $5$ and  $SE(2)$-symmetry index $3$.}
    {\includegraphics[width=6cm]{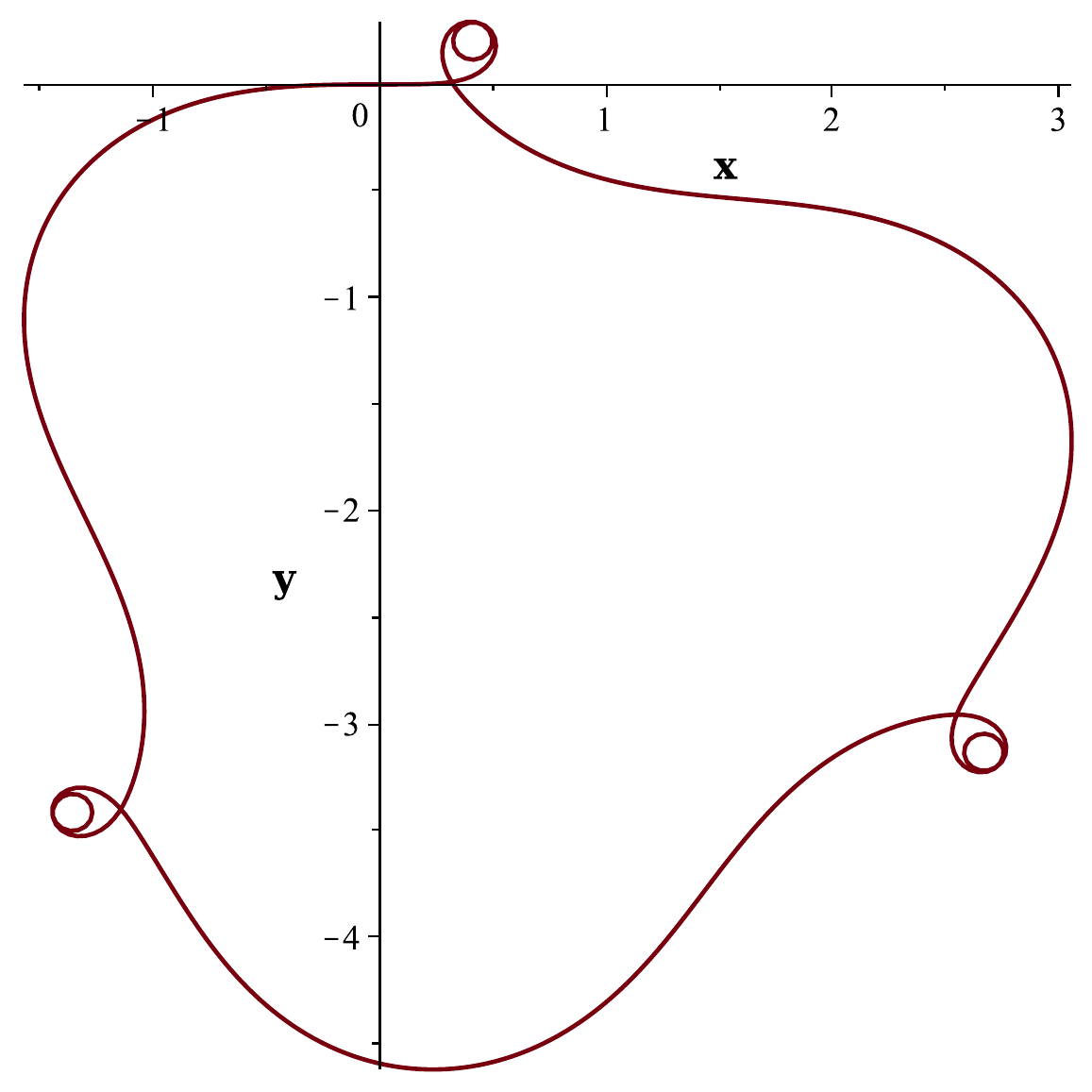}} \\ \vspace{1cm}
    \subcaptionbox{A curve with curvature $\kappa_{7/3}$, $s\in [0,14\pi]$  has turning number $3$ and  $SE(2)$-symmetry index $7$.}
    {\includegraphics[width=6cm]{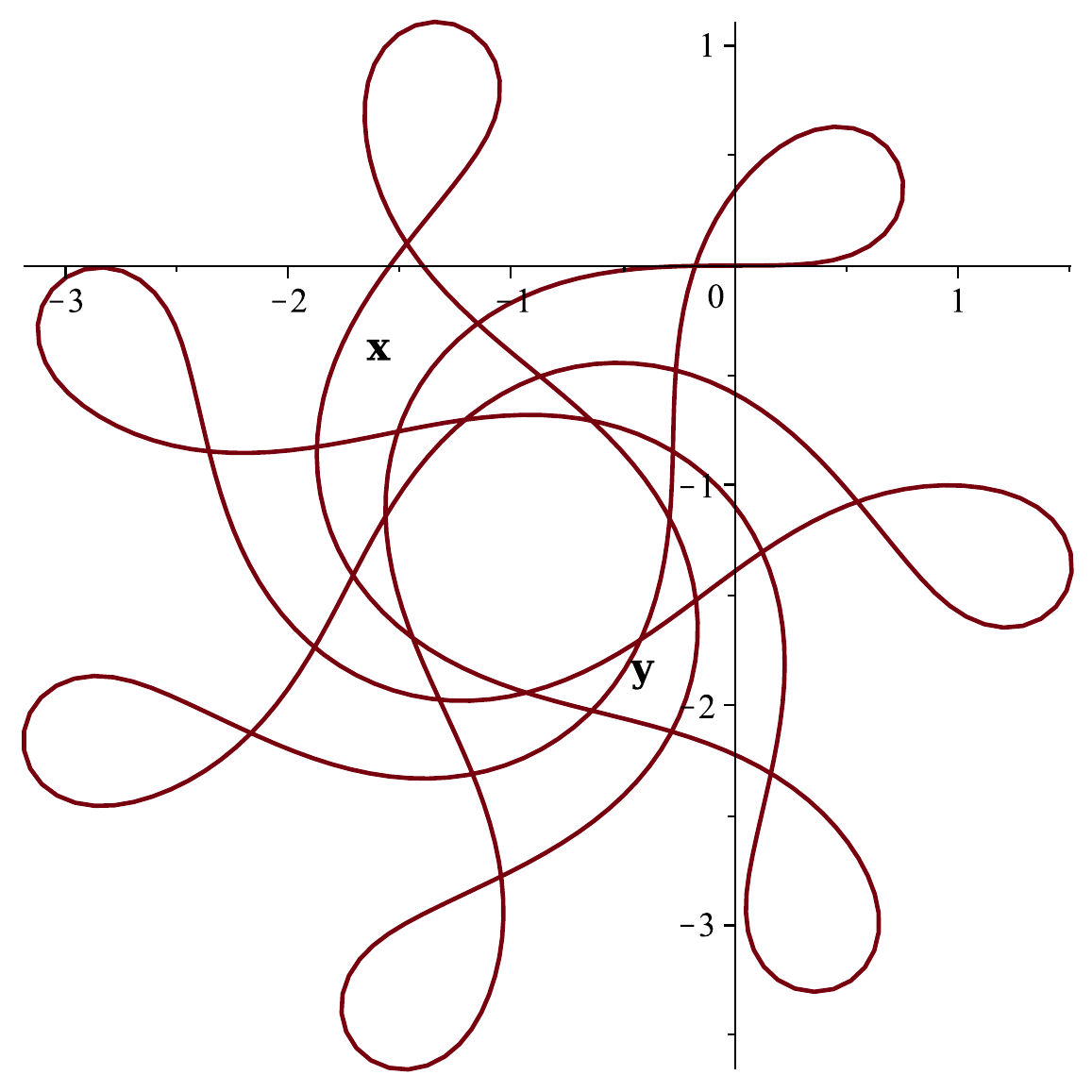}} \hspace{1cm}
    \subcaptionbox{A curve with curvature $\kappa_{-5/3}$, $s\in [0,10\pi]$  has turning number $-3$ and  $SE(2)$-symmetry index $5$.}
    {\includegraphics[width=6cm]{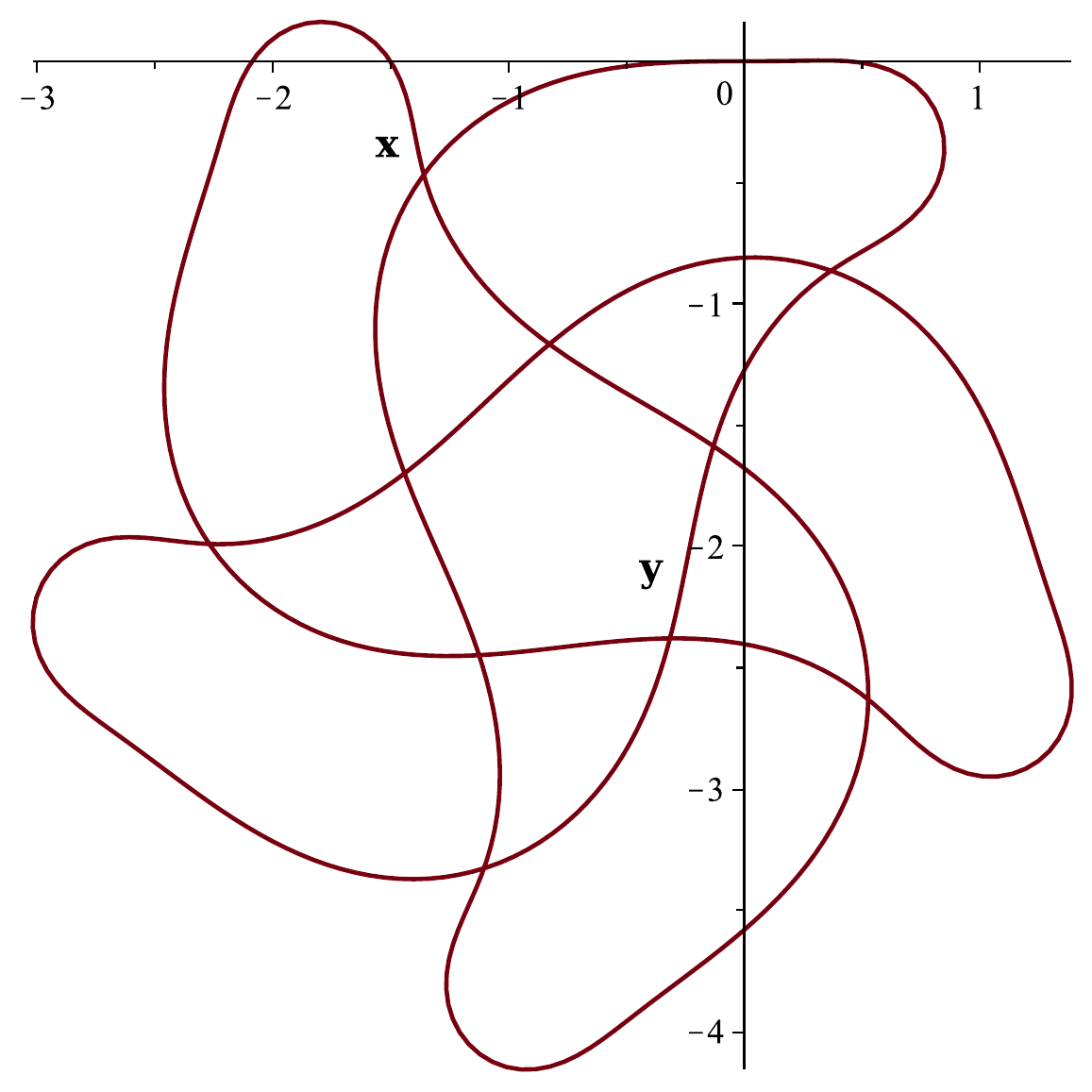}}
    \caption{Closed curves reconstructed from $\kappa_r(s)$, where $r$ is a rational number.}
    \label{fig-k-fraction}
  \end{figure}

  It is worth noting  that if, in formula \eqref{eq-kn},  we replace the integer $n$ with a rational number $r=\frac{q}{\xi}$ such that $q\neq 1$ and $\xi$ are relatively prime  then  by Lemma \ref{lem-closed},  a curve reconstructed from $\kappa_r(s)$, $s\in [0,2\pi q]$ will be a closed curve with the $SE(2)$-symmetry index $q$ and turning number $\xi$. See Figure \ref{fig-k-fraction} for examples.

 \end{ex}
\section{Affine reconstruction}\label{sect-arecon}
In this section, we start by showing how Picard iterations can be used to reconstruct a curve from its affine curvature. We proceed by 
proving  some  upper bounds  related to Picard iterations and using them to estimate how close, relative to the Hausdorff  distance, two curves can be brought together by a special affine transformation, provided the affine curvature functions  of the curves are $\delta$-close in the $L^{\infty}$-norm (Theorem~\ref{thm-aff-est}).

 \newpage
\begin{theorem}[Affine reconstruction]\label{thm-aff-rec} Let $\mu(\al)$ be a continuous function on an interval $[0,L]$. Then there is a unique, up to an special affine transformation, curve $\cgam$ with the affine arc-length parametrization  $\gamma(\al)=\left(x(\al),y(\al)\right)$, $\al\in [0,L]$, such that $\mu(\al)=x'' (\al)y'''(\al)-y''(\al)x'''(\al)$ is its affine curvature function. 
\end{theorem}
\begin{proof} According to  \eqref{aTN}, \eqref{ta} and \eqref{na},  $\gamma$ is a solution of the following system of  first order differential equations: 
\begin{align}
\label{gamma'T} \gamma'(\al)&=T(\al)\\
\label{T'N} T'(\al)&=N(\al)\\
\label{N'T} N'(\al)&=-\mu(\al) T(\al),
\end{align}
(equivalent to a third order ODE system of  two decoupled equations $\gamma'''=-\mu\gamma'$).
Due to well known results on the existence and uniqueness of solutions to linear ODEs (see Theorems 5 and 6, Section 13.3 in \cite{Nagle}), there exists a unique solution of  \eqref{gamma'T}-\eqref{N'T} with the initial data 
\beq\label{in-data} \gamma(0)=(0,0),\quad  T(0)=(1,0), \quad N(0)=(0,1).\eeq
Let $\gamma_0(\al)$ be such a solution  parametrizing a  curve  $\cgam_0$.  Let $\cgam_1$ be another curve with the affine arc-length parametrization  $\gamma_1(\al)$, $\al\in [0,L]$, such that $\mu(\al)$ is its affine curvature.
Let $T_1=\gamma_1'(0)$ and $N_1=\gamma_1''(0)$. Then there exists a unique special affine transformation $g\in SA(2)$ which is a composition of a translation by the vector $-\gamma_1(0)$, followed by the unimodular  linear transformation  $\begin{pmatrix} T_1(0)\\N_1(0)\end{pmatrix}^{-1}$, such that
$$g\cdot\gamma_1(0)=(0,0), \quad  g\cdot T_1=(1,0), \quad g\cdot N_1=(0,1).$$
Since $\mu$ and $d\al$  are { $SA(2)$-invariant}, it follows that the curve $g \,\cgam_1$ parametrized by $g\gamma_1$ satisfies\eqref{gamma'T}-\eqref{N'T} with the same initial   data \eqref{in-data} and, therefore, $\cgam_0= g\,\cgam_1$.
\end{proof}
 \begin{figure}[h!]
    \centering
    \subcaptionbox{$\mu = 0$. \label{fig-parabola}}
    {\includegraphics[width=4.7cm]{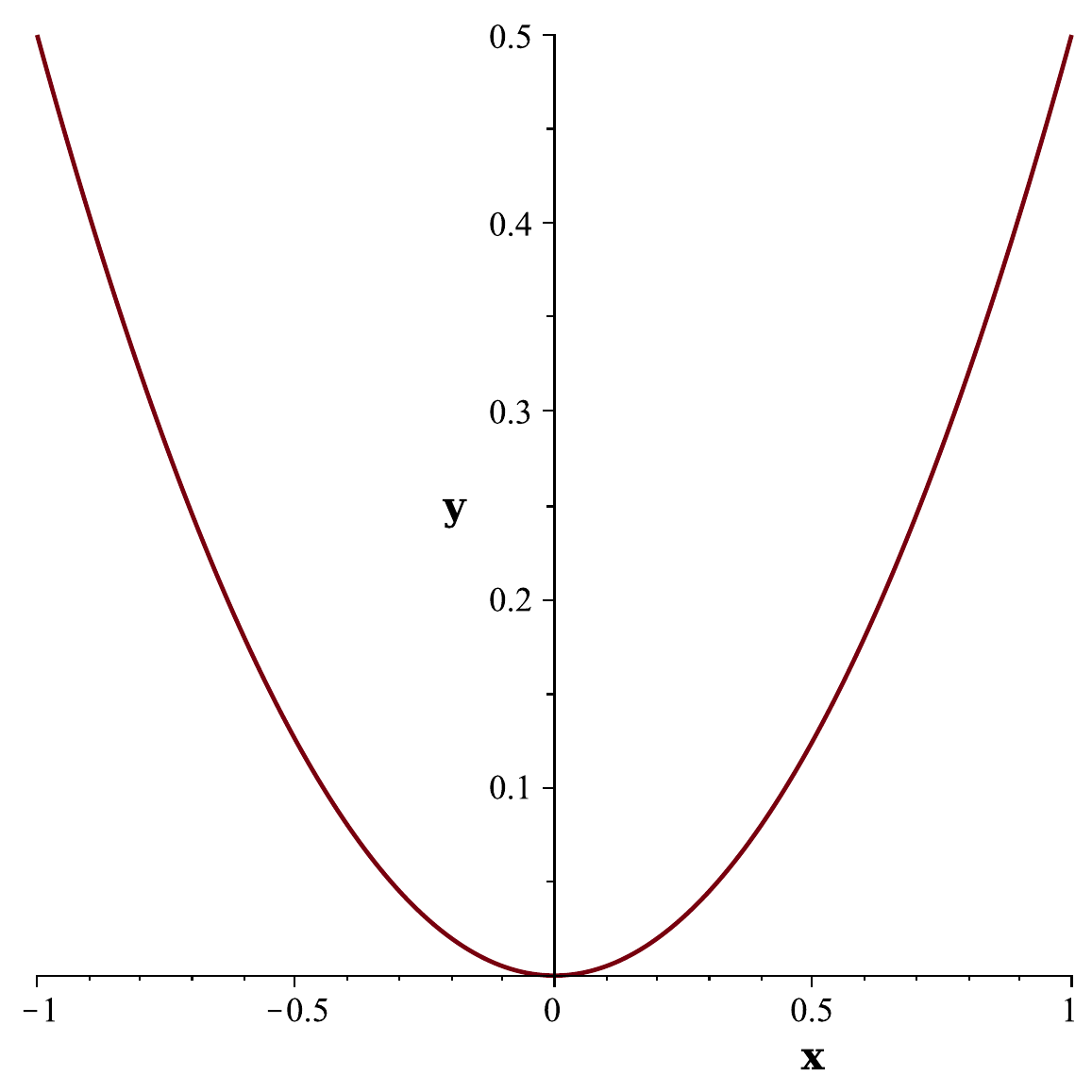}} \hspace{.5cm}
    \subcaptionbox{{ $\mu = 3$.} \label{fig-ellipse}}
    {\includegraphics[width=4.7cm]{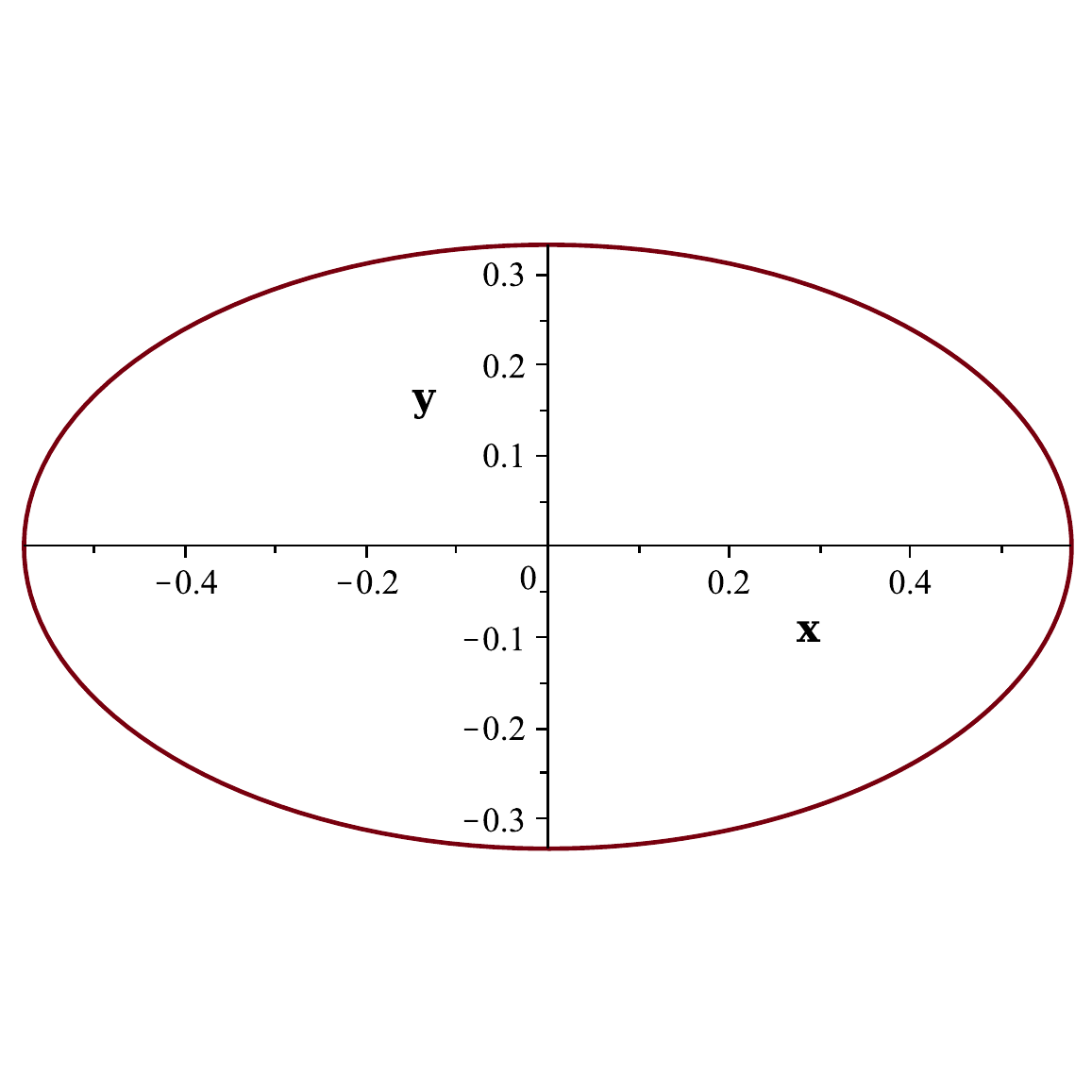}} \hspace{.5cm}
    \subcaptionbox{$\mu = -3$. \label{fig-hyperbola}}
    {\includegraphics[width=4.7cm]{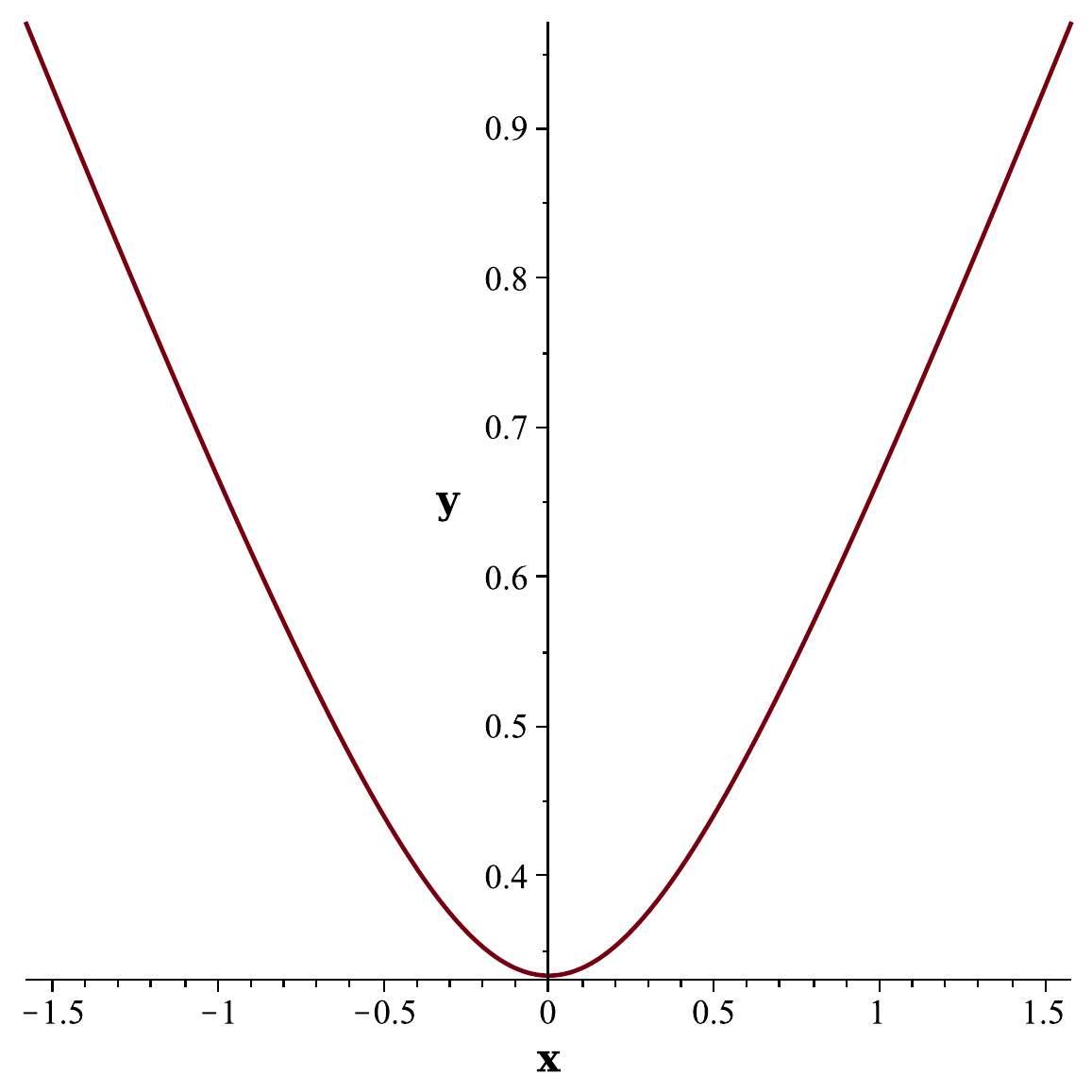}}
  \caption{Examples of curves with constant special affine curvature functions.}
    \label{mu-const} 
  \end{figure}

We now consider computational aspects of reconstruction of  a curve from its  affine curvature.  Once $T(\al)$ is known, $\gamma$ can be reconstructed by integration which can be done exactly  or numerically depending on the complexity of $T(\al)$.   To find $T(\al)$, one needs to solve the system  \eqref{T'N}-\eqref{N'T}.

When $\mu(\al)$ is a constant function, standard methods can be applied.
In fact, as shown in  \cite{gugg}, if   $\mu = 0$  then the reconstructed curve, with the initial conditions \eqref{in-data}, is a parabola $\gamma = (\al,\frac12 \al^2)$. When $\mu > 0$ $$\gamma = \left(\frac{\sin(\sqrt{\mu}\al)}{\sqrt\mu}, -\frac{\cos(\sqrt\mu \al)}{\mu}\right)$$ is an ellipse.
When $\mu < 0$ $${\gamma = \left(\frac{\sinh(\sqrt{-\mu}\al)}{\sqrt{-\mu}}, -\frac{\cosh(\sqrt{-\mu} \al)}{\mu}\right)}$$ is a hyperbola. See Figure~\ref{mu-const} for specific examples.

When $\mu$ is   non-constant but analytic one can use  power series methods to find the solutions. The power series solutions for the case when $\mu$ is a monomial: $\mu=c\al^k$  are given in the Appendix.  For an arbitrary continuous  function $\mu$, we approximate $T(\al)$ by applying \emph{Picard iterations} as follows. 

As discussed in Section~\ref{ssect-amf}, equations  \eqref{T'N} and \eqref{N'T} are equivalent to the matrix equation \eqref{eq-Aa}, where $A(\al)=\begin{pmatrix}{}
    T(\al) \\
  N(\al)  \\
\end{pmatrix}$ is the affine frame matrix  and $C(\al)$ is the affine Cartan matrix given by \eqref{cartan-ma}. 
 The Picard iterations are defined as:
 \begin{align}
\nonumber  A_0(\al)&=A_0\\
 \label{an} A_n(\al)&=A_0+\int_0^\al C(t)A_{n-1}(t)dt, \text{ for } n>0. 
 \end{align}
 It is well known that on any interval $[0,L]$, as $n\to\infty$ the sequence of $\seq{A_n(\al)}$ uniformly converges to the unique  matrix of continuous functions $A(\al)$, satisfying  
 the integral equation
  \beq  A(\al)=A_0+\int_0^\al C(t)A(t)dt\eeq and, therefore,  the differential equation \eqref{eq-Aa} with the initial value $A_0$. 
A direct proof  for the convergence of \eqref{an} to the solutions of  \eqref{eq-Aa} with the initial value $A_0$, where $C$  is an arbitrary continuous matrix, is given in \cite{gugg} Lemma 2-12. 

\begin{ex}
  \begin{figure}[h!]
    \centering
{\includegraphics[width=6cm]{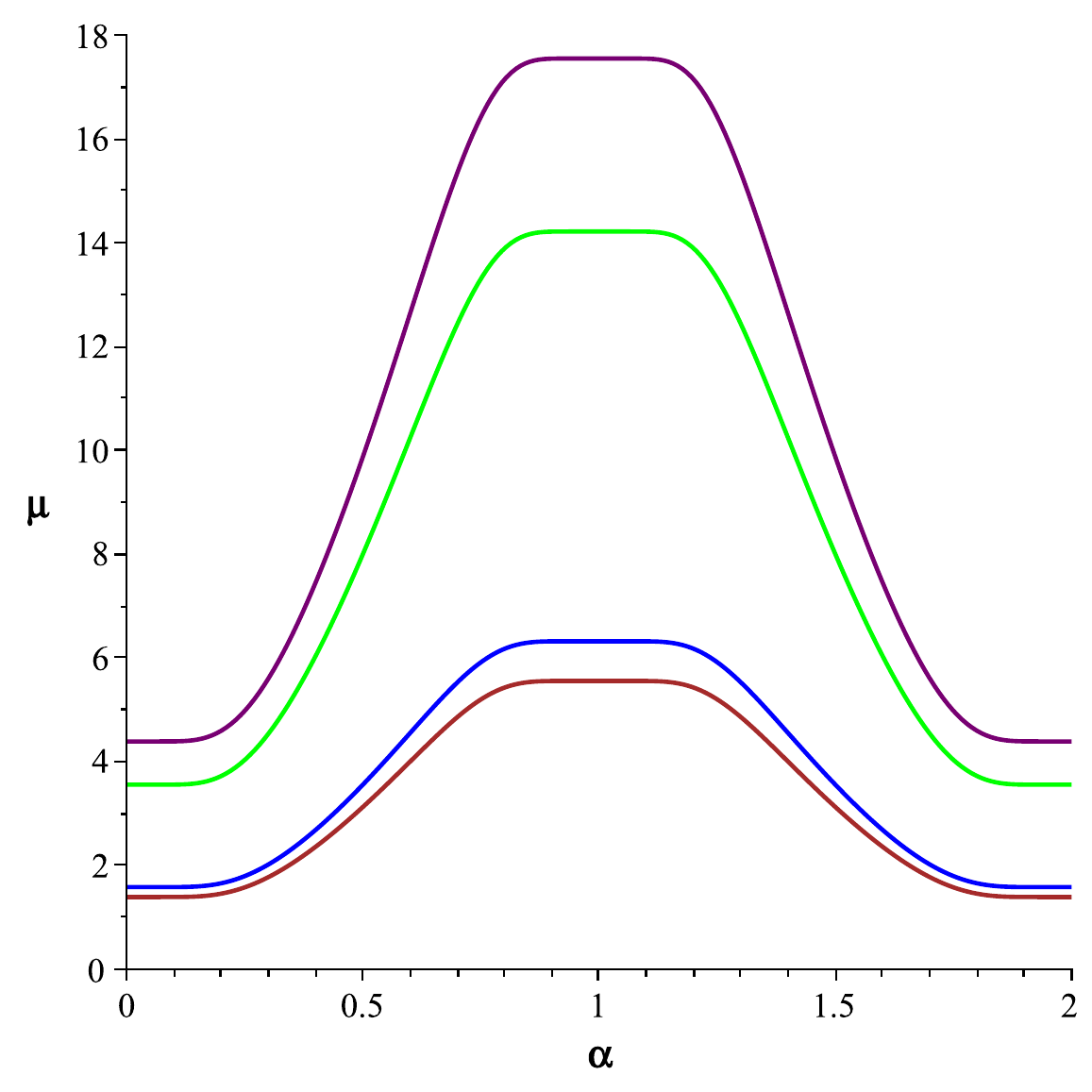}}
    \caption{ \textcolor{darkmagenta}{$\mu^*_{2/3}(\al)$}, \textcolor{olivedrab}{$\mu^*_{3/5}(\al)$}, \textcolor{darkblue}{$\mu^*_{2/5}(\al)$}, and \textcolor{brown}{$\mu^*_{3/8}(\al)$}, $\alpha\in [0,2]$ given by \eqref{eq-mu-n}.}
    \label{fig-mu*}
  \end{figure}

  \begin{figure}[h!]
    \centering
    \subcaptionbox{Approximation of a curve with equi-affine curvature \textcolor{darkblue}{$\mu_{2/5}$} on $[0,22]$.}
    {\includegraphics[width=6cm]{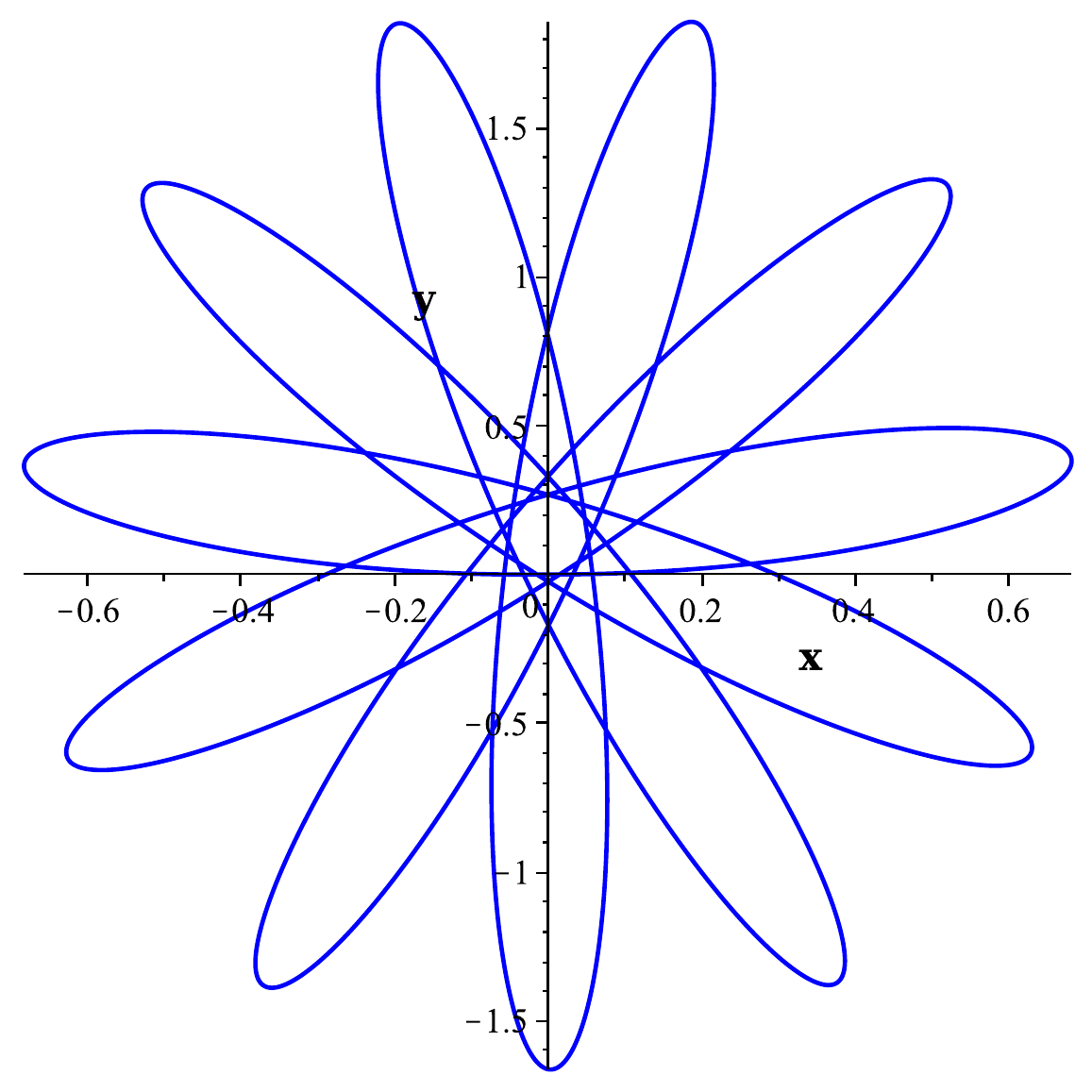}} \hspace{1cm}
    \subcaptionbox{Approximation of a curve with equi-affine curvature \textcolor{olivedrab}{$\mu_{3/5}$} on $[0,20]$.}
    {\includegraphics[width=6cm]{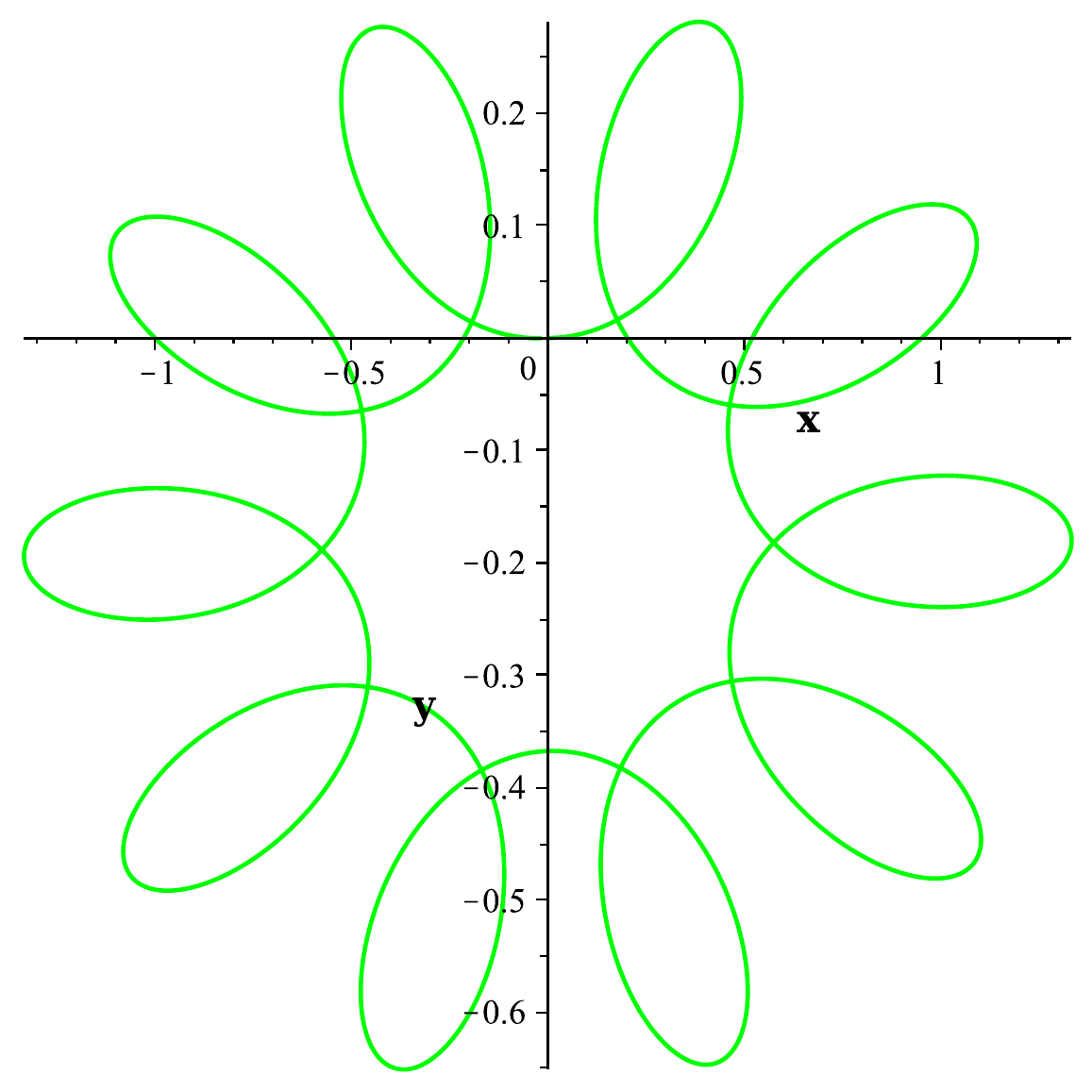}} \\ \vspace{1cm}
    \subcaptionbox{Approximation of a curve with equi-affine curvature \textcolor{darkmagenta}{$\mu_{2/3}$} on $[0,10]$.}
    {\includegraphics[width=6cm]{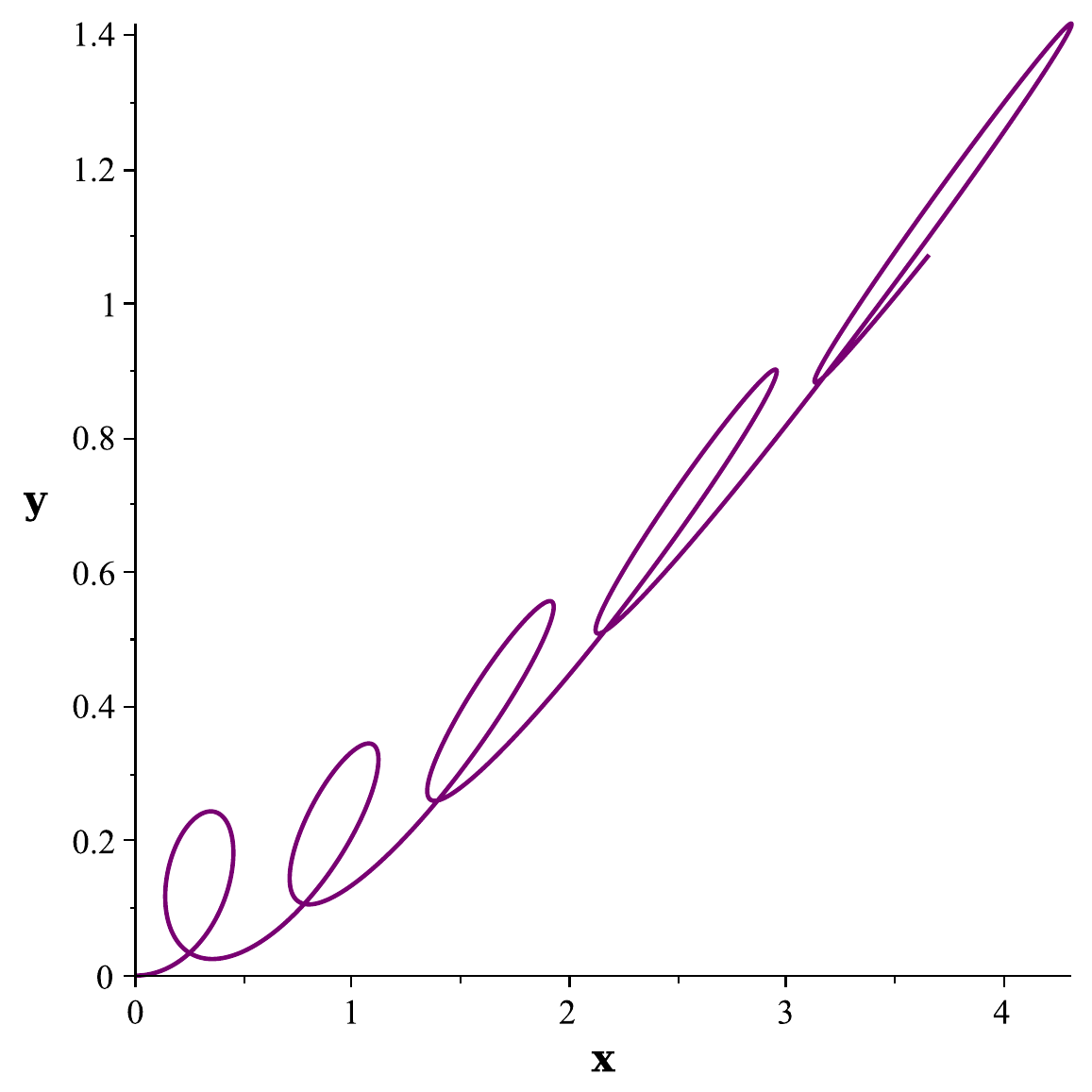}} \hspace{1cm}
    \subcaptionbox{Approximation of a curve with equi-affine curvature  \textcolor{brown}{$\mu_{3/8}$} on $[0,8]$.}
    {\includegraphics[width=6cm]{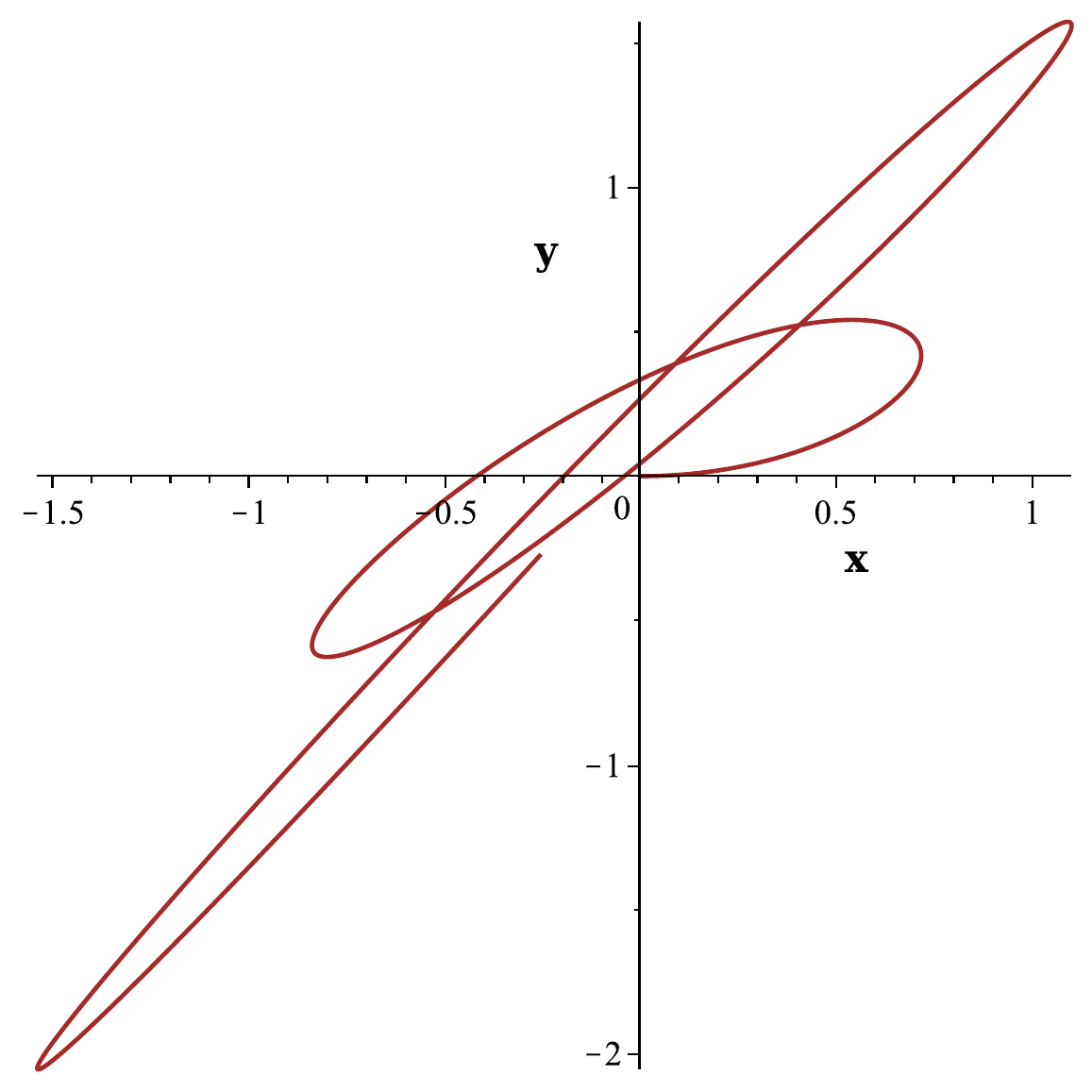}}
    \caption{Approximations of curves, using 200 Picard iterations,  reconstructed from periodic extensions of the affine curvature functions shown in Figure~\ref{fig-mu*}.}
    \label{fig-picard}
  \end{figure}
    We will briefly look at a few curves that are reconstructed from their affine curvatures.
    Recall the bump function $f(s)$ given by \eqref{eq-bump}.
    Let
    \begin{equation}\label{eq-mu-n}
      \mu_n^*(\al) = n^2\pi^2(f(\al)+1)^2
    \end{equation}
    with domain $[0,2]$ and let $\mu_n(\al)$ be the periodic extension of $\mu_n^*$ to $\mathbb{R}$.
    In Figure \ref{fig-picard}, we show approximations (using 200 Picard iterations) of curves with affine curvatures $\mu_{2/3}$, $\mu_{2/5}$, $\mu_{3/5}$, and $\mu_{3/8}$, initial conditions $\gamma(0)=(0,0)$, and $A_0=I$.

     It is important to note that the affine analog to Lemma \ref{lem-closed} is not valid. Indeed, it is shown, for instance, in Example 7.2 in \cite{kogan-olver2003}, that in contrast with the Euclidean case, the total special affine curvature $\int \mu d\al$ of a closed curve is not topologically invariant, and thus it cannot be used to determine whether the curve is closed or open. Moreover, as remarked in \cite{Verpoort} on p.~421,  there does not exist a function of $\mu$  whose integral is a topological invariant.   With this in mind, it is worth noting that the approximations of the curves with special affine curvatures $\mu_{2/5}$ and $\mu_{3/5}$ appear to be closed, while the curves with the affine curvature functions  $\mu_{2/3}$ and $\mu_{3/8}$ show no sign that they would close if their domain was extended.

\end{ex}

  We now investigate  the ``closeness'' of two curves reconstructed from ``close'' affine curvatures. We start by establishing certain upper bounds:
  \begin{lemma} Assume that $||C||_{[0,L]}=\max\{1, ||\mu||_{[0,L]}\}=c$. Let $A_n$ be defined by the Picard iterations  \eqref{an} and $A$ be the limit of these iterations. Then  for any $\al\in [0,L]$ the following inequalities hold:{
  \begin{align}
 \label{bound-n} \maxc{A_n}(\al)&\leq\maxc{A_0}\sum_{i=0}^n \frac{(\bc\al)^i}{i!},\\
 \label{bound-a}    \maxc{A}(\al)&\leq \maxc{A_0} e^{\bc\al},\\
 \label{bound-n-n-1}       \maxc{A_n-A_{n-1}}(\al)& \leq\maxc{A_0}\frac{(\bc\al)^n}{n!},\\
 \label{bound-n-a}   \maxc{A_n-A}(\al)&\leq \maxc{A_0} e^{\bc\al} \frac{(\bc\al)^{n+1}} {(n+1)!}.
\end{align}
}
\end{lemma}
  \begin{proof}

  \begin{enumerate}
  \item For $n=0$, \eqref{bound-n} states that $\maxc{A_0}\leq \maxc{A_0}$, which is trivially true.  We proceed by induction. Assume that \eqref{bound-n} holds for all $0\leq k<n$.
  Then  from \eqref{an}, \eqref{ineq} and the triangle inequality,  we have
  \beq\label{pf-eq-an1} \maxc{A_n}(\al)\leq \maxc{A_0}+ \int_0^\al \maxc{CA_{n-1}}(t)dt.\eeq
  Note that  for any matrix $A=\begin{pmatrix} a_{11} &a_{12}\\a_{21} &a_{22}\end{pmatrix}$, we have
  $CA=\begin{pmatrix} a_{21} &a_{22}\\ -\mu a_{11} &-\mu a_{12}\end{pmatrix}$ and, therefore, since $c\geq  ||\mu||_{[0,L]}$ and $c\geq 1$,
  \beq\label{boundCA} \maxc{CA}(t)\leq \bc \maxc{A}(t).\eeq
  Returning to \eqref{pf-eq-an1} and using  the inductive assumption, we then have 
  \begin{align}\label{pf-eq-an2}& \maxc{A_n}(\al)\leq \maxc{A_0}+ \bc\int_0^\al \maxc{A_{n-1}}(t)dt\leq \maxc{A_0}+\bc  \maxc{A_0}\sum_{i=0}^{n-1}\int_0^\al \frac{(\bc t)^{i}}{i!}\,dt\\
  \nonumber  &=\maxc{A_0}\left( 1+\bc\, \sum_{i=0}^{n-1} \frac{\bc^i \al^{i+1}}{(i+1)!}\right)
= \maxc{A_0}\left(1+ \sum_{i=1}^{n} \frac{\bc^i \al^{i}}{i!}\right)=\maxc{A_0}\sum_{i=0}^n \frac{(\bc\al)^i}{i!}.
\end{align}
\item To show \eqref{bound-a}, we use \eqref{eq-lim2} and \eqref{bound-n}
\beq\label{lim-abs}\maxc{A}(\al)=\lim_{n\to \infty} \maxc{A_n}(\al)\leq\maxc{A_0}\sum_{i=0}^\infty \frac{(\bc\al)^i}{i!}=\maxc{A_0}e^{\bc\al}.\eeq
\item For $n=1$, \eqref{bound-n-n-1} states that $\maxc{A_1-A_0}(\al)\leq \maxc{A_0} \bc\al$. This, indeed, holds because by \eqref{an} $  A_1(\al)-A_{0}(\al)=\int_0^\al C(t) A_0 dt$, and so by \eqref{ineq} and  \eqref{boundCA}
$$\maxc{A_1-A_0}(\al)\leq\int_0^\al \left<C(t) A_0\right>dt\leq \int_0^\al \bc \maxc{A_0} dt=\maxc{A_0}\bc\al.$$
We proceed by induction. Assume that \eqref{bound-n-n-1} holds for all $1\leq k<n$.  By \eqref{an},  
$$ A_n(\al)-A_{n-1}(\al)=\int_0^\al C(t) (A_{n-1}(t)-A_{n-2}(t)) dt,$$
and then  by \eqref{ineq},  \eqref{boundCA}, and the inductive hypothesis
 $$ \maxc{A_n-A_{n-1}}(\al) \leq\int_0^\al \bc \maxc{A_{n-1}-A_{n-2}}(t) dt\leq \maxc{A_0}\int_0^\al \bc  \frac{(\bc t)^{n-1}}{(n-1)!} dt =\maxc{A_0}\frac{(\bc\al)^n}{n!}.$$

  \item  To show \eqref{bound-n-a}, we note that for any integer $j>0$, due to the triangle inequality  and  \eqref{bound-n-n-1}, we have
 \begin{align} \label{pf-eq-an-a} \maxc{A_n-A}(\al)&\leq \maxc{A_n-A_{n+1}}(\al)+ \maxc{A_{n+1}-A_{n+2}}(\al)+\dots\\
 \nonumber&+\maxc{A_{n+j-1}-A_{n+j}} (\al)+\maxc{A_{n+j}-A} (\al)\leq \maxc{A_0} \sum_{i={n+1}}^{n+j}\frac{(\bc\al)^i}{i!}+\maxc{A_{n+j}-A} (\al).
 \end{align}
 Since $A_{n+j}(\al)$ converges   to $A(\al)$ as $j\to\infty$, $\lim_{j\to\infty}\maxc{A_{n+j}-A} (\al)=0$, and so \eqref{pf-eq-an-a} implies
 \beq \maxc{A_n-A}(\al)\leq \maxc{A_0}\sum_{i={n+1}}^{\infty}\frac{(\bc\al)^i}{i!}=\maxc{A_0}\left(e^{c\al}-\sum_{i={0}}^{n}\frac{(\bc\al)^i}{i!}\right).\eeq
 
 Due to Taylor's remainder theorem, there exists $\al_0\in[0,\al]$, such that 
 $$R_n=e^{c\al}-\sum_{i={0}}^{n}\frac{(\bc\al)^i}{i!}= e^{c\al_0}\frac{(c\al)^{n+1}}{(n+1)!}\leq e^{c\al}\frac{(c\al)^{n+1}}{(n+1)!},$$
  where the last inequality is true because $c>0$ and so $e^{c\al}$ is an increasing function.
\end{enumerate}
  
    \end{proof}
    Next, we establish  bounds on the distance between two affine frames reconstructed from two $\delta$-close (in the $L^\infty$  norm) affine  curvature functions.\footnote{This result is consistent with a well known ODE result {on} continuous dependence of the solutions of an ODE on its parameters (see, for instance, Theorem 10, Section 13.4 in \cite{Nagle} and Theorem 3, Chapter 5 in \cite{birkoff-rota-1962}).}  
   \begin{proposition} Let $\mu(\al)$ and $\tilde \mu(\al)$ be two continuous functions  on the interval $[0,L]$ and let $C$ and $\tC$ be corresponding Cartan's matrices defined by \eqref{cartan-ma}.  Let $\hat c=\max\{1, ||\mu||_{[0,L]},||\tilde \mu||_{[0,L]}\}$.  Let $A_n$  and $\tA_n$ be defined by the Picard iterations  \eqref{an} for the given matrices $C$ and $\tC$, respectively,  and $A$,  $\tA$ be the limits of these iterations.   
   If $||\mu-\tilde \mu||_{[0,L]}\leq \delta$, then for all $\al\in [0,L]$:  
  {
   \begin{align}   
   \label{atan} \maxc{A_n-\tA_n}(\al)&\leq\maxc{A_0} \delta\al \sum_{i=0}^{n-1} \frac{(\hc\al)^i}{i!} \qquad\text{  for } n>0,\\
 \label{ata} \maxc{A-\tA}(\al)&\leq \maxc{A_0}\delta\al e^{\hat c\al}.
   \end{align}
   }
    \end{proposition}
 \begin{proof} 
 \begin{enumerate}
 \item We first observe that for all $\al\in [0,L]$,  $\maxc{C-\tC}(\al)=|\mu(\al)-\tilde\mu(\al)|<\delta$.
 
  For $n=1$, \eqref{atan} states that $\maxc{A_1-\tA_1}(\al)\leq \maxc{A_0}\delta\al$. This, indeed, holds because by \eqref{an}, keeping in mind that $A_0(\al)=\tA_0(\al)=\maxc{A_0}$, we have 
 \begin{align*}  \maxc{A_1-\tA_{1}}(\al)&\leq \int_0^\al \maxc{(C-\tC)A_0}(t) dt\leq \int_0^\al \maxc{A_0}|\tilde\mu(t)-\mu(t)| dt\\
 &\leq \maxc{A_0} \int_0^\al\delta dt=\maxc{A_0}\delta\al.
 \end{align*} 
 We proceed by induction. Assume that \eqref{atan} holds for all $1\leq k<n$, then 
  \begin{align}\label{pf-ta-l1}&\maxc{A_n-\tA_n}(\al)\leq \int_0^\al \maxc{CA_{n-1}-\tC\tA_{n-1}}(t) dt\\
 \nonumber  &= 
  \int_0^\al \maxc{CA_{n-1}- C\tA_{n-1}+C\tA_{n-1}-\tC\tA_{n-1}}(t) dt\\
  \label{pf-ta-l2} &\leq  \int_0^\al \hc \maxc{A_{n-1}- \tA_{n-1}} (t) dt +   \int_0^\al \delta \maxc{\tA_{n-1}}(t) dt\\
\label{pf-ta-l3} &  \leq \int_0^\al \hc     \maxc{A_0} \delta t \sum_{i=0}^{n-2} \frac{(\hc t)^i}{i!}dt +\int_0^\al \maxc{A_0} \delta  \sum_{i=0}^{n-1} \frac{(\hc t)^i}{i!}dt\\
\nonumber &=\maxc{A_0}\delta\left( \hc \sum_{i=0}^{n-2} \frac{\hc^i\al^{i+2}}{i!(i+2)}+  \sum_{i=0}^{n-1} \frac{\hc^i\al^{i+1}}{(i+1)!}\right)\\
\nonumber& =
\maxc{A_0}\delta\left(\sum_{i=1}^{n-1} \frac{\hc^i\al^{i+1}}{(i-1)!(i+1)}+  \sum_{i=0}^{n-1} \frac{\hc^i\al^{i+1}}{(i+1)!}\right)\\\
\nonumber  &= \maxc{A_0}\delta\left(  \al+\sum_{i=1}^{n-1}\hc^i\al^{i+1}\left( \frac{1}{(i-1)!(i+1)}+\frac{1}{(i+1)!}\right)\right)\\
\nonumber &= \maxc{A_0}\delta \al  \left(1+ \sum_{i=1}^{n-1}\hc^i\al^{i}\frac{1}{i!}\right)= \maxc{A_0}\delta \al\sum_{i=0}^{n-1}\frac{(\hc\al)^{i}}{i!},
  \end{align}
where, in line \eqref{pf-ta-l1}, we use  \eqref{an}, \eqref{ineq}, and the triangle inequality.  In line  \eqref{pf-ta-l2}, we use  \eqref{boundCA} and  the triangle inequality, and  in line  \eqref{pf-ta-l3}, we use  the inductive assumption and \eqref{bound-n}.

\item To show \eqref{ata}, we note that since  $A_n(\al)$ and  $\tA_n(\al)$ converge   to $A(\al)$  and $\tA(\al)$, respectively,   as $n\to\infty$,  then by \eqref{eq-lim2},
$\lim_{n\to \infty} \maxc{A_n-\tA_n}(\al)= \maxc{A-\tA}(\al)$, 
and  so taking the limit of both sides in the inequality \eqref{atan} as $n\to\infty$, we obtain \eqref{ata}. 

\end{enumerate}
 \end{proof}
In the next theorem, we establish an upper bound on how close (in the Hausdorff distance)  two curves  with $\delta$-close (in the $L^\infty$-norm) affine  curvature functions can be brought together by a special affine transformation.  
 \begin{theorem}[Affine estimate]\label{thm-aff-est} Let $\cgam_1$ and $\cgam_2$ be two  $C^3$-smooth planar curves of the same affine arc-length $L$. Assume $\mu_1(\al)$ and $\mu_2(\al)$, $\al\in [0,L]$ are their respective affine curvature functions. Assume further that $\cgam_2$  satisfies the initial conditions \eqref{in-data}\footnote{If we omit this assumption, then the right-hand side of \eqref{eq-gt1t2} must be multiplied by $\maxc{A_2(0)}$  according to \eqref{ata}, and so the right-hand side of  \eqref{eq-h-set} must be multiplied by $\maxc{A_2(0)}$, as well.}.  If    $||\mu_1-\mu_2||_{[0,L]}\leq \delta$ and $\hat c=\max\{1, \infin{\mu_1},\infin{\mu_2}\}$, then there is $g\in SA(2)$, such that
   \beq\label{eq-h-set} d(g\,\cgam_1, \cgam_2)\leq \sqrt  2\,\frac{\delta L}{\hat c} (e^{\hat cL}-1), \eeq
 where $d$ is the Hausdorff distance. 
 \end{theorem}
 \begin{proof}
 For $i=1,2$, let $\gam_i(\al)$, $\al\in [0,L]$ be the affine-arc length parameterization of $\cgam_i$, while $T_i(\al)=\gamma_i'(\al)$ and  $N_i(\al)=\gamma_i''(\al)$ are the affine frame vectors along the corresponding curves.
 Then, there is a unique $g\in SA(2)$, such that 
 \beq\label{g-data} g\gamma_1(0)=\gamma_2(0) = (0,0),\quad  g T_1(0)=T_2(0) = (1,0), \quad gN_1(0)=N_2(0) = (0,1).\eeq
 Due to the $SA(2)$-invariance of the affine curvature function, the curve $g\,\cgam_1$ parametrized by $g\gamma_1(\al)$ has affine curvature function   $\mu_1(\al)$. 
 It follows from Theorem~\ref{thm-aff-rec}, that $g\gamma_1(\al)$ is the unique solution of \eqref{gamma'T}-\eqref{N'T}, with $\mu(\al)=\mu_1(\al)$ and $\gamma_2(\al)$ is the unique solution of \eqref{gamma'T}-\eqref{N'T}, with $\mu(\al) = \mu_2(\al)$, both with initial conditions \eqref{g-data}. 

Denote the affine frame of $g\,\cgam_1$ as  $A(\al)=\begin{pmatrix} gT_1(\al)\\ gN_1(\al)\end{pmatrix}$ and the affine frame of $\cgam_2$ as $\tilde A(\al)=\begin{pmatrix} T_2(\al)\\ N_2(\al)\end{pmatrix}$. Then 
 \beq \label{eq-gt1t2}\maxc{gT_1-T_2}(\al)\leq \maxc{A-\tilde A}(\al)\leq \delta\al e^{\hat c\al}, \eeq
 where the first inequality is due to the definition of $\maxc{\cdot}$ and the second inequality is due to  \eqref{ata}.
 Since $g\gamma_1(\al) = \int_0^\al{ g} T_1(t)dt + T_0$ and $\gamma_2(\al) = \int_0^\al T_2(t)dt + T_0$, we have for all $\al \in [0,L]$:
 \begin{align}
\label{ineq-g12} {  \maxc{g\gamma_1-\gamma_2}(\al)}\leq \int _0^\al \maxc{ gT_1- T_2}(t)dt\leq  \int _0^\al \delta t e^{\hat ct} dt
\leq  \int _0^\al \delta L e^{\hat ct} dt=\frac {\delta L} {\hc}(e^{\hat c\al}-1).
\end{align}
It then follows  from \eqref{hcg} and \eqref{ineq-g12} that
\begin{align*}
d(g \, \cgam_1,\cgam_2) \leq\sqrt{2} \infin{ g\gamma_1-\gamma_2}{ =} \sqrt 2 { \max_{\al\in[0,L]} \maxc{g\gamma_1-\gamma_2}(\al)}
  \leq \sqrt{2} \frac{\delta L}{\hat c} (e^{\hat cL}-1).
\end{align*}

 \end{proof}

%
\section{Conclusion}\label{sect-concl}
In this paper, we considered practical aspects of reconstructing planar curves with prescribed Euclidean or affine curvatures. An immediate extension of the current work would be the reconstruction of planar curves with prescribed projective curvatures, and obtaining distance estimates between curves, modulo a projective transformation, compared to the distance between the projective curvatures. Indeed, the projective group, containing both the special Euclidean and the special affine groups, plays a crucial  role in computer vision (see, for, instance \cite{faugeras-2001} and  \cite{hartley04}). Extension to space curves is another direction with immediate applications. 

By considering specific group actions, we take advantage of their specific structural properties and obtain results that can be immediately suitable  for applications. However, the  generalization of the moving frame method by Fels and Olver, \cite{FO99, olver2015}, allows us, in principle, to generalize our approach to an action of an arbitrary Lie group $G$ on curves (or even on higher dimensional submanifolds) in some ambient metric space. In such a generalization, a $G$-equivariant moving frame map from the corresponding jet space  to the group  $G$ plays the role of  the $G$-frame matrix $A$, appearing in this paper, and we will seek an estimate of how close two submanifolds can be brought together by an element of $G$, provided the Maurer-Cartan invariants  for the $G$-action are sufficiently close.

In this paper, we used the Hausdorff distance  between curves when considering both the $SE(2)$- and the $SA(2)$-actions on the plane. However, while  the Hausdorff distance is $SE(2)$-invariant, it is \emph{not}   $SA(2)$-invariant and so it does not provide a natural measure of distance between two curves in  the special affine case.
In a future work, it is worthwhile to explore $SA(2)$-invariant alternatives for measuring distance between two curves, based, for instance, on  the area of the region between two curves. In the generalization to other group actions,  the goal would be to consider a $G$ -invariant distance between two submanifolds.
\section{Appendix}

If a given special affine curvature is analytic, it is possible to reconstruct the corresponding curve by looking for power series solutions to the second order ODE system $T_{\al\al} = -\mu(\al)T$. We illustrate this approach  by reconstructing curves whose special affine curvatures are of the form $\mu(\al) = c\al^k$ for $c\in \R$ and  $k \in \mathbb{N}$.
\begin{proposition} {  For $c\in \R$, $k \in \mathbb{N}$ and $T_0, N_0\in \R^2$, such that $\det[T_0, N_0] = 1$,  let $\mathcal C$  be the curve  whose affine curvature function is  $\mu(\al) = c\al^k$, the initial affine tangent vector is $T_0$ and the initial affine normal is $N_0$.  Then the affine tangent vector along $\mathcal C$ is given by 
the absolutely convergent power series}
\beq\label{series}T(\al)=-T_0\Gamma\left(-\frac 1 K\right)\sum\limits_{i=1}^\infty \frac{(-c)^i\al^{Ki}}{i!K^{2i+1}\Gamma\left(-\frac 1 K+i+1\right)}+N_0\Gamma\left(\frac 1 K\right)\sum\limits_{i=1}^\infty \frac{(-c)^i\al^{K(i+1)}}{i!K^{2i+1}\Gamma\left(\frac 1 K+i+1\right)},\eeq
where $K=k+2$ and $\Gamma$ denotes the gamma function.
  
\end{proposition}

\begin{proof}
We first represent the tangent vector $T(\al)$ by

\begin{equation}
    T = b_0 + b_1\al + b_2\al^{2} + b_3\al^{3} + \cdots + b_n\al^{n} \cdots
\end{equation}
where each $b_i$ is a vector coefficient, with  $b_0=T_0$ and $b_1=N_0$ being the initial values of the affine tangent and the affine normal, respectively. 

We write out  the power series representation of $T_{\al\al}$ and $-c\al^k T$:
\begin{equation}
  T_{\al\al} = 0b_0 + 0b_1\al + 2b_2 + 3\cdot 2 b_3\al + \cdots + n(n-1) b_n\al^{n-2} \cdots 
\end{equation}
\begin{equation}
    -c\al^k T = -cb_0\al^k - cb_1\al^{(k+1)} - cb_2\al^{(k+2)} - cb_3\al^{(k+3)} - \cdots - cb_n\al^{(k+n)} - \cdots
\end{equation}

The equality of these  two power series implies the equality of vector-coefficients with the same powers of  $\al$ in two series.  It follows that 
\beq\label{eq-bi} b_n = \begin{cases}0 & \text{ when } 2 \leq n \leq k+1\\
         -\dfrac{cb_{n-(k+2)}}{n(n-1)}& \text{ when } n \geq k+2
\end{cases}.\eeq
 Then $b_{k+2}$ and $b_{k+3}$ can be written in terms of $b_0$ and $b_1$:%
\begin{equation}
  b_{k+2} = -\frac{cb_0}{(k+2)(k+1)}, \qquad b_{k+3} = -\frac{cb_1}{(k+3)(k+2)}.
\end{equation}
Using induction,  when $n\mod (k+2) = 0$, we can express $b_n$  in terms of $b_0$, when $n\mod (k+2) = 1$, we  can express $b_n$  in terms of $b_1$, and we can show that  otherwise $b_n = 0$.
This gives us the power series representation for $T$ in terms of $b_0$ and $b_1$ as
\begin{align}
 \nonumber T(\al) = b_0 + b_1\al + 
 \sum\limits_{i=1}^\infty (-c\al^{k+2})^i \Bigg( & \left( \prod_{j=1}^{i} \frac{1}{j(k+2)\left(j\left(k+2\right) -1\right) } \right)b_0 + \\
  & \left( \prod_{j=1}^{i} \frac{1}{j\left(k+2\right)\left(j\left(k+2\right) + 1\right) } \right)b_{1} \al \Bigg). \label{eq:T}
\end{align}

We can split (\ref{eq:T}) into two parts:
\begin{align}
\label{b0_series}
B_0=&b_0\sum\limits_{i=1}^\infty  \left(\prod_{j=1}^{i} \frac{1}{j(k+2)(j(k+2) -1) } \right) (-c\al^{k+2})^i \\
\nonumber &= {b_0\sum\limits_{i=1}^\infty  \left(\prod_{j=1}^{i} \frac{1}{(j(k+2) -1) } \right) \frac{(-c\al^{k+2})^i}{i!(k+2)^i}=b_0\sum\limits_{i=1}^\infty \Psi_-(K,i)\frac{(-c\al^K)^i}{i!K^i}}
\end{align}

and 

\begin{align}
  B_1=&\label{b1_series}
  b_1\sum\limits_{i=1}^\infty  \left(\prod_{j=1}^{i} \frac{1}{j(k+2)(j(k+2) + 1) } \right)(-c\al^{k+2})^i \al \\
  \nonumber &=b_1\sum\limits_{i=1}^\infty  \left(\prod_{j=1}^{i} \frac{1}{(j(k+2) +1) } \right) \frac{(-c\al^{k+2})^i\al}{i!(k+2)^i}=b_1\sum\limits_{i=1}^\infty \Psi_+(K,i)\frac{(-c\al^K)^i\al}{i!K^i},
\end{align}
 where $K=k+2$ and 
\begin{align}
 \label{psi_-} \Psi_-(K,i)&=\prod_{j=1}^{i} \frac{1}{(jK -1) }=\frac 1 {K^i}\frac{1}{\prod_{j=1}^{i}(j -\frac 1 K) }\\
  \label{psi_+}\Psi_+(K,i)&=\prod_{j=1}^{i} \frac{1}{(jK +1) }=\frac 1 {K^i} \frac{1}{\prod_{j=1}^{i}(j +\frac 1 K) }.
\end{align}
These functions involve what is called \emph{rising factorials}, defined by
$$z^{\bar i}:=z(z+1)\cdots(z+i-1)=\prod_{j=0}^{i-1} (z+j).$$
Rising factorials  can be expressed in terms of $\Gamma$ functions, $\Gamma(z)=\int_0^\infty x^{z-1}e^{-x}dx$,  as
$$z^{\bar i}=\frac{\Gamma(z+i)}{\Gamma(z)}.$$
For details see formulas (5.84), (5.85) and (5.89) on pp. 210-211 of \cite{Graham94}.
Since
\begin{align}
\prod_{j=1}^{i}\left(j -\frac 1 K\right)=-K\left(-\frac 1 K\right)^{\overline{i+1}}=-K\frac{\Gamma\left(-\frac 1 K+i+1\right)}{\Gamma\left(-\frac 1 K\right)},\\
\prod_{j=1}^{i}\left(j +\frac 1 K\right)=K\left(\frac 1 K\right)^{\overline{i+1}}=K\frac{\Gamma\left(\frac 1 K+i+1\right)}{\Gamma\left(\frac 1 K\right)},
\end{align}
we can rewrite \eqref{psi_-}-\eqref{psi_+} using $\Gamma$ functions:
\begin{align}
 \Psi_-(K,i)&=\prod_{j=1}^{i} \frac{1}{(jK -1) }=-\frac 1 {K^{i+1}}\frac{\Gamma\left(-\frac 1 K\right)}{\Gamma\left(-\frac 1 K+i+1\right)},\\
  \Psi_+(K,i)&=\prod_{j=1}^{i} \frac{1}{(jK +1) }=\frac 1 {K^{i+1}}\frac{\Gamma\left(\frac 1 K\right)}{\Gamma\left(\frac 1 K+i+1\right)}.
\end{align}
Therefore, 
\begin{align}
\nonumber  B_0(\al)=&b_0\sum\limits_{i=1}^\infty \Psi_-(K,i)\frac{(-c\al^K)^i}{i!K^i} \\
\nonumber&=b_0\sum\limits_{i=1}^\infty \left(-\frac 1 {K^{i+1}}\frac{\Gamma\left(-\frac 1 K\right)}{\Gamma\left(-\frac 1 K+i+1\right)}\right)\frac{(-c\al^K)^i}{i!K^i} \\
&=-b_0\Gamma\left(-\frac 1 K\right)\sum\limits_{i=1}^\infty \frac{1}{\Gamma\left(-\frac 1 K+i+1\right)}\frac{(-c\al^K)^i}{i!K^{2i+1}}, \\
\nonumber B_1(\al)=&b_1\sum\limits_{i=1}^\infty \Psi_+(K,i)\frac{(-c\al^K)^i\al}{i!K^i} \\
\nonumber&=b_1\sum\limits_{i=1}^\infty \left(\frac 1 {K^{i+1}}\frac{\Gamma\left(\frac 1 K\right)}{\Gamma\left(\frac 1 K+i+1\right)}\right)\frac{(-c\al^K)^i\al}{i!K^i} \\
&=b_1\Gamma\left(\frac 1 K\right)\sum\limits_{i=1}^\infty \frac 1{\Gamma\left(\frac 1 K+i+1\right)}\frac{(-c\al^K)^i\al}{i!K^{2i+1}}.
 \end{align}

Convergence of series  \eqref{eq:T} for all $\al$ follows from a general known result (Theorem 39.22 p.560 \cite{TenenbaumPollard-1963}).
Directly, absolute convergence of sub-series  (\ref{b0_series}) and (\ref{b1_series}) can be verified  by the ratio test, implying absolute convergence of series  \eqref{eq:T}.
\end{proof}

The power series for the affine arc-length parameterization $\gamma(\al)$ is  obtained by integrating the  series $T(\al)$. 
See Figures \ref{fig:mu-alpha} and \ref{fig:mu-alpha2} for reconstructions of curves with curvatures $\mu(\al) = \al$ and $\mu(\al) = \al^2$ respectively.

  \begin{figure}
\begin{minipage}{.45\linewidth}
    \centering
    \includegraphics[width=\linewidth]{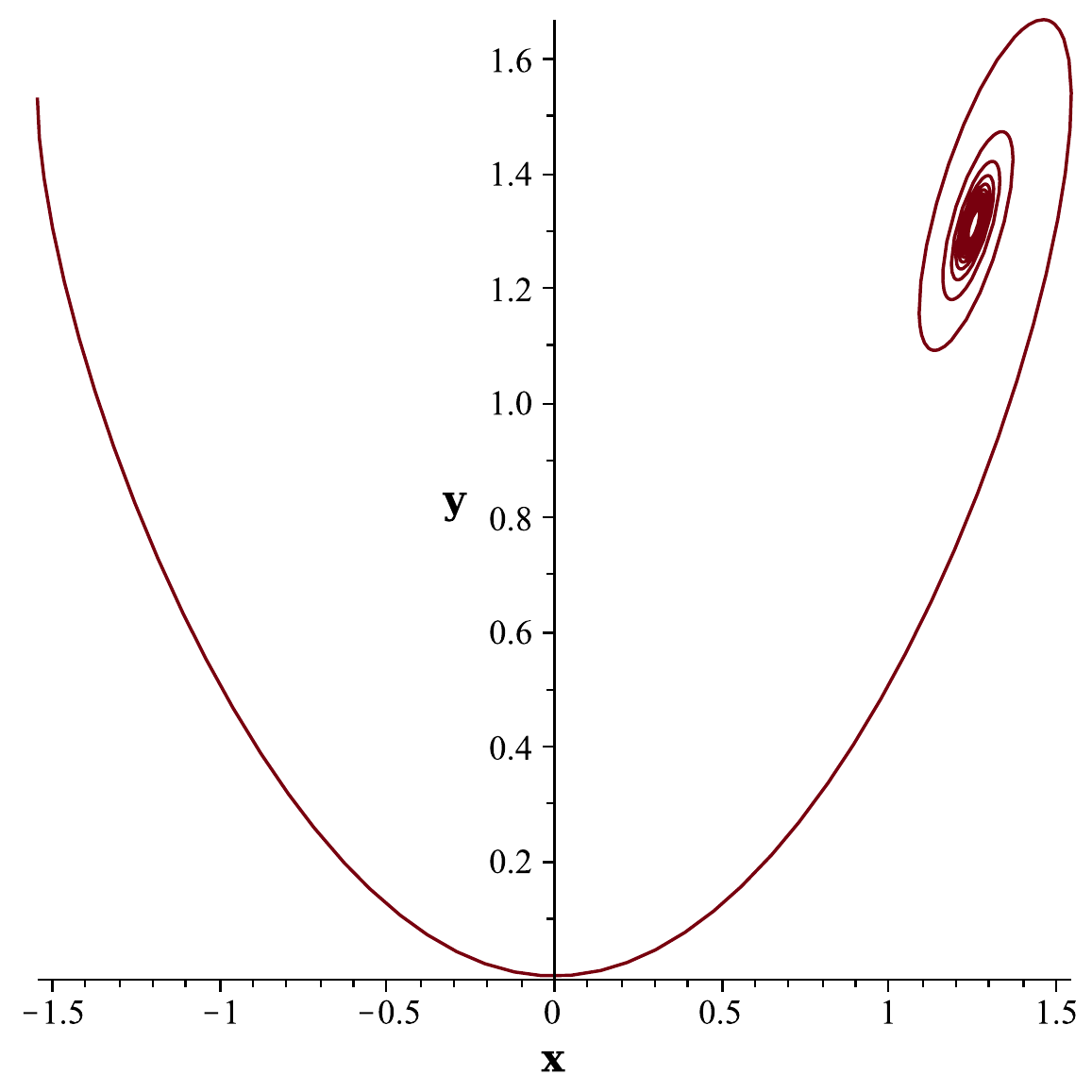}
    \caption{Curve with special affine curvature $\mu(\al) = \al$}
    \label{fig:mu-alpha}
\end{minipage}
\hspace{.05\linewidth}
\begin{minipage}{.45\linewidth}
    \centering
    \includegraphics[width=\linewidth]{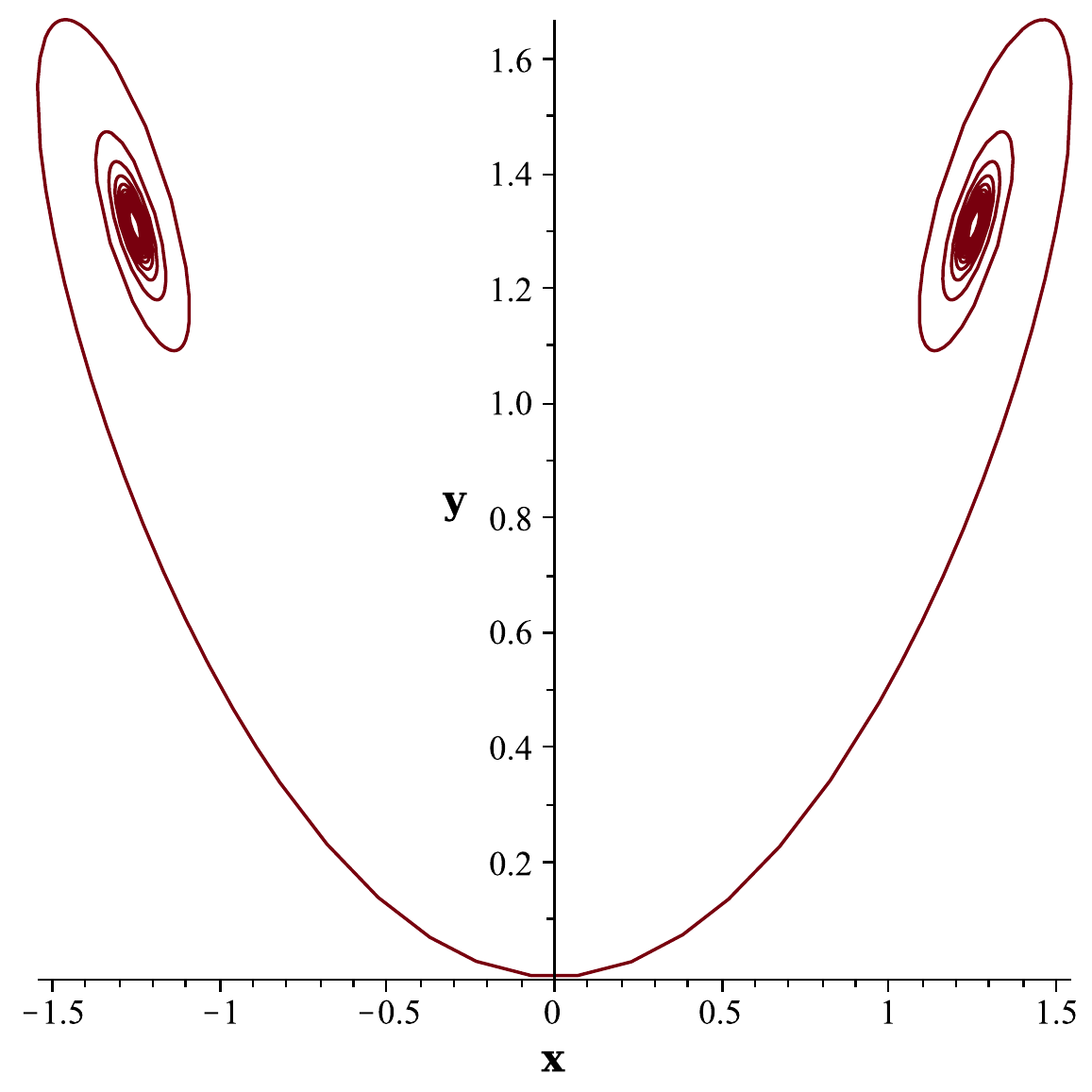}
    \caption{Curve with special affine curvature $\mu(\al) = \al^2$}
    \label{fig:mu-alpha2}
\end{minipage}
  \end{figure}

\begin{remark} The system $T_{\al\al} = -c\al^kT$ consists of two decoupled  equations of the type  $u''(\alpha) = -c\al^k u(\alpha)$, whose general solution in terms of 
 of  the Bessel functions, can be found, for instance, in Section 14.1.2, subsection 7, number 3 of  \cite{polyanin-zaitsev-2012}. 	The Bessel functions can be expended into power series involving the  gamma function, recovering series \eqref{series}. The advantage of formula \eqref{series} is in its explicit dependence on the initial vectors $T_0$ and $N_0$. In addition, our direct proof illustrates how the power series approach can be applied for other analytic affine curvatures $\mu(\alpha)$.
\end{remark}

\paragraph{Acknowledgement:} This work was performed during the REU 2020 program at the North Carolina State University (NCSU) and was supported by the Department of Mathematics at NCSU and the NSA grant H98230-20-1-0259.  At the time when the project was performed,  Jose Agudelo was an undergraduate student at North Dakota State University, Brooke Dippold was an undergraduate student at Longwood University, Ian Klein was an undergraduate student at Carleton College, Alex Kokot was an undergraduate student at the University of Notre Dame, and Eric Geiger was a graduate student at NCSU. Irina Kogan is a Professor of Mathematics at NCSU.   The project was mentored by Eric Geiger and Irina Kogan. A poster based on this project received a honorable mention at JMM 2021.

\bibliographystyle{amsplain}
\providecommand{\bysame}{\leavevmode\hbox to3em{\hrulefill}\thinspace}
\providecommand{\MR}{\relax\ifhmode\unskip\space\fi MR }
\providecommand{\MRhref}[2]{%
  \href{http://www.ams.org/mathscinet-getitem?mr=#1}{#2}
}
\providecommand{\href}[2]{#2}

\end{document}